\documentclass[12pt]{amsart}

\usepackage{euscript,eufrak,amscd,verbatim}

\theoremstyle{plain}
\newtheorem{theorem}{Theorem}[section]
\newtheorem{lemma}[theorem]{Lemma}
\newtheorem{proposition}[theorem]{Proposition}
\newtheorem{definition}[theorem]{Definition}

\newtheorem{remark}[theorem]{Remark}

\newtheorem{corollary}[theorem]{Corollary}

\def\R{\mathbb{R}}

\def\Z{\mathbb{Z}}

\def\vp{\varepsilon}

\def\u{\underline}
\def\o{\overline}
\def\om{\omega}
\def\bb{\mathbb}

\def\limsup{\operatornamewithlimits{\overline{\text{\rm{lim}}}}}
\def\liminf{\operatornamewithlimits{\underline{\text{\rm{lim}}}}}

\renewcommand{\epsilon}{\varepsilon}
\renewcommand{\phi}{\varphi}
\DeclareMathSymbol{\varnothing}{\mathord}{AMSb}{"3F}
\renewcommand{\emptyset}{\varnothing}

\def\bf{\textbf}
\def\cal{\mathcal}

\begin{document}

\date{}

\title{Pressures for multi-potentials in semigroup dynamics}

\author{Eugen Mihailescu}

\maketitle

\begin{abstract}

On compact metric spaces $X$ we introduce and study the amalgamated pressure,  condensed pressure,   trajectory pressure,  exhaustive pressure, and their capacities  on  $\mathcal C(X, \mathbb R^m)$, for non-compact sets $Y \subset X$,  with respect to finitely generated semigroups $G$ acting on  $X$.    Properties and comparisons between these pressures, and relations with inverse limits are given. Next we introduce  a measure-theoretic amalgamated entropy, and prove  a Partial Variational Principle for amalgamated pressure.  Local amalgamated and local exhaustive entropies are studied for measures on $X$, and they provide estimates for amalgamated  and exhaustive topological entropies. For marginal ergodic measures $\nu$ on $X$, we find an upper bound for the  local exhaustive entropy. These notions can be used for classification of semigroup actions.  We  apply  the amalgamated pressure of unstable multi-potentials, also  to estimate the dimensions of slices in $G$-invariant saddle-type sets. 
\end{abstract}

\



 \section{Introduction.}
 We introduce and explore several new nonlinear functionals of pressure and their capacities on non-compact sets in compact metric spaces $X$, for continuous multi-potentials $\Phi \in \mathcal C(X, \R^m)$,  with respect to a semigroup $G$ of maps on $X$ with  finite set of generators $G_1$. These help to develop the ergodic theory for  the case when several transformations $f_1, \ldots, f_m$ act simultaneously on $X$ and at each step $n \ge 1$, the continuous observables $\phi_1, \ldots, \phi_m$ are evaluated  on the arbitrary trajectory $\{x, f_{j_1}(x), f_{j_2}\circ f_{j_1}(x), \ldots, f_{j_{n}}\circ\ldots\circ f_{j_1}(x)\}$ of $x \in X$ in a way which depends on the path $C = (j_1, \ldots, j_n) \in \{1, \ldots, m\}^n$. 
 
 In ergodic theory, the notion of topological pressure for a single potential $\phi$ plays an important role, and its theory is now well developed (for eg \cite{Bo-carte},  \cite{KH}, \cite{Ru1}, \cite{Wa}). In \cite{Bo0} Bowen also  studied the topological entropy for general non-compact sets, and in particular in the case of  Borel sets consisting of the generic points for invariant probability measures on $X$. 
  A general unifying notion for various types of entropy, pressure and dimensions on non-compact sets, which has found many applications in dynamics and dimension theory, is that of  Carath\'eodory-Pesin structures (\cite{Pe}). 
 
It appeared moreover the need for a notion of entropy for a (semi)group of maps, for example in the case of a foliation on a manifold and a pseudo-group of holonomy maps, such as the one studied by Ghys, Langevin and Walczak in  \cite{GLW}. 
Various types of entropy and topological pressure for the dynamics of semigroups, and their properties together with connections with the Carath\'eodory-Pesin structures,  were studied by many authors.  

 In particular, the dynamics for finitely generated semigroup actions on compact metric spaces, and its connections with other fields, were investigated in these papers.  Bi\'s \cite{Bi} and Hofmann and Stoyanov \cite{HS} studied a type of entropy for finitely generated semigroups, which involves intersections of the Bowen balls along all  trajectories of order $n$ of arbitrary points $x \in X$. In \cite{Bu} Bufetov defined a type of entropy for free semigroup actions, by using averages of cardinalities of spanning sets over all trajectories, and by using also a skew product transformation. This idea was then extended to a notion of pressure for real-valued potentials for free semigroup actions, whose properties were studied by Carvalho, Rodrigues and Varandas in \cite{CRV} and by Lin, Ma and Wang in \cite{LMW}. 
 In a sense, semigroup dynamics has similarities also with the  dynamics along prehistories of  non-invertible maps.  A notion of preimage entropy which takes in consideration the preimage sets was studied by Hurley in \cite{Hu}. 

Relations between topological pressure, entropy, Lyapunov exponents, and Hausdorff dimension for quasicircles and Julia sets, were first studied  by Bowen  (\cite{Bo1}, and Ruelle (\cite{Ru2}).  The dimensions of stable and unstable intersections with hyperbolic sets of diffeomorphism were studied by Manning and McCluskey in \cite{MM}. And in \cite{Y} Young found a formula for the dimension of a hyperbolic invariant measure for a surface diffeomorphism. This study was then developed for  diffeomorphisms with non-uniform hyperbolicity  by several authors, for eg in Pesin \cite{Pe}. Also relations between pressure, preimages distributions, and  dimensions for endomorphisms were studied in \cite{M-MZ} and by Mihailescu and Urba\'nski  \cite{MU-BLMS}. 
 
\
 
 In the sequel, we explore several new types of  pressure and their capacities on non-compact sets for the multi-potential $\Phi = (\phi_1, \ldots, \phi_m) \in \mathcal C(X, \R)^m$ with respect to the semigroup $G$ of endomorphisms of $X$ generated by  $G_1 = \{id_X, f_1, \ldots, f_m\}$, in order to study the simultaneous dynamics  of  $f_1, \ldots, f_m$ on $X$ when the observables $\phi_1, \ldots, \phi_m$ are evaluated depending on the random trajectory in $X$.

 Thus,  we introduce the \textbf{amalgamated pressure}, the \textbf{condensed pressure}, the \textbf{pressure along a set $\cal Y$ of trajectories}, and the \textbf{exhaustive pressure}, together with their respective \textbf{capacities} on non-compact subsets $Y \subset X$. These nonlinear functionals will be defined in general by employing Carath\'eodory-Pesin structures (\cite{Pe}). However, one of the functionals above, namely the amalgamated pressure over a set $\mathcal Y$ of trajectories, is not necessarily  a Carath\'eodory-Pesin structure. These different types of pressure functionals correspond to taking various trajectories using the  maps $f_j, 1 \le j \le m$. We have finitely many  systems represented by the maps $f_j, 1 \le j \le m$, on a space of states $X$, and $m$ observables  (coordinates of position, momentum, spin, etc, see for eg \cite{CE}) given by the  potentials $\phi_j, 1 \le j \le m$. The consecutive sums $S_n \Phi(x, \om)$ that we use, give the combined value along the trajectory $\om$ of the observable $\phi_{\om_k}$ measured after $f_{\om_k}$ has been applied, for $1 \le k \le n$. 
 The trajectories determined by semigroups can be encountered also in the dynamics of non-invertible maps, for local inverses and prehistories (for eg \cite{FM}, \cite{M-MZ}, \cite{MU-BLMS}). 
  
 Some of our notions of pressure (like the amalgamated pressure, the exhaustive pressure, and pressure over a set of trajectories) are different from previous notions  in the literature, while other 
notions (like the condensed pressure) extend previous notions of pressure for single potentials.  Our notions of pressure are different in general from the pressure for free semigroups studied in \cite{Bu}, \cite{CRV}, \cite{LMW}. When all potentials $\phi_j$ are equal,  the above notions apply to single real-valued potentials. In this case our condensed pressure extends the notion studied in \cite{Bi}, \cite{HS}. 

  As an example, one can think of a random walk on the lattice $\mathbb Z^2$, given by four independent transformations $f_1, \ldots, f_4$ defined by going up, down, left or right, and observables $\phi_1, \ldots, \phi_4$  on $\Z^2$. For  $C = (j_1, \ldots, j_n) \in\{1, \ldots, 4\}^n$, the combined value of  observables on a $C$-trajectory is $\phi_{j_1}(x) + \phi_{j_2}(f_{j_1}(x)) + \ldots + \phi_{j_n}(f_{j_{n-1}}\circ \ldots \circ f_{j_1}(x)),$
  which is the consecutive sum that we  use in the sequel. If $\phi_1(x), \ldots, \phi_4(x)$ represent the amount gained/lost when going from $x$ up, down, left or right respectively,  this combined value is the total amount  after walking on $C$. This is similar to a random walk in random scenery, introduced  by Borodin \cite{Bor} and  Kesten and Spitzer \cite{KS}.

The above pressures have also  connections  with random dynamics (\cite{Ar}, \cite{KS}), dynamics of non-invertible maps (\cite{M-MZ}-\cite{MU-JFA}, \cite{N}), and abelian  semigroups of toral endomorphisms (\cite{Be}).
 
 \

 To summarize our results, in \textbf{Section 2} the notions of amalgamated, condensed, and exhaustive pressures and their capacities are introduced for multi-potentials $\Phi$ with respect to finitely generated semigroups $G$ on the compact metric space $X$. Also a notion of pressure along arbitrary sets of infinite trajectories is introduced.
 
 In \textbf{Section 3}, we  give relations between these types of pressure and their capacities, and we \textbf{compare} between them. 
  For instance the lower  exhaustive pressure $P_l^+$ will be shown to be smaller than the amalgamated pressure $P^A$, which in turn will be shown to be smaller than a notion of free pressure $P_{free}$, which in turn is smaller than our upper condensed pressure $P_u$. Thus, these  notions of pressure are needed when subtle distinctions need to be made in the long-term behaviour of the system. They can be used also for classification of semigroup actions, up to conjugacy.
  
Then we study these notions in examples, namely for  \textbf{semigroups of toral endomorphisms}. Our results are also related to the problem of invariant closed sets in $\mathbb T^k, k \ge 2$ for commutative semigroups of toral endomorphisms studied in \cite{Be}.
 
 Next in \textbf{Section 4}, we consider the \textbf{inverse limit} $\hat X_{G_1}$ of the dynamics of the semigroup $G$ generated by   $G_1$, and connect it with our notions. 
 We explore relations between the amalgamated pressure and the exhaustive pressure of a multi-potential $\Phi$, and the pressure of the associated potential $\Phi^+$ on the lift  $\Sigma_m^+\times X$.
 
 In \textbf{Section 5}, we adopt the view that the most significant ``invariant'' measures for the semigroup action are the $F$-invariant probability measures on $\Sigma_m^+ \times X$ (see for eg \cite{Ar}). Indeed, measures that are preserved by all elements of $G$ may be very few or may not even exist; also, stationary measures on $X$ with respect to some shift-invariant measure $\rho$ on $\Sigma_m^+$ may not contain all the information. Thus, we introduce a  measure-theoretic amalgamated entropy $h^A(\mu, G_1)$ for the projection on $X$ of an $F$-invariant probability measure $\mu$, which takes into consideration the combined actions on $X$ of all the maps from $G$. Using this type of entropy $h^A(\mu, G_1)$, we  prove a \textbf{Partial Variational Principle} for amalgamated pressure $P^A(\Phi)$ of multi-potentials on $X$.

 Then in \textbf{Section 6}, by analogy with the local entropy introduced by Brin and Katok (\cite{BK}), we define  \textbf{local amalgamated entropies} $h^u_\mu(x, G_1), h^l_\mu(x, G_1)$ of probability measures $\mu$ on $X$. We prove that, if we can estimate them on a Borel set $Y$ of positive $\mu$-measure, this gives the amalgamated  entropy $h^A$ on $Y$. We also introduce the \textbf{local exhaustive entropy} of $\mu$, $h^+_\mu(x, G_1)$. Moreover, for \textbf{marginal measures} $\nu$ on $X$ of $F$-invariant ergodic probability $\hat \mu$, we find an upper bound for the local exhaustive entropy $h^+_\nu(x, G_1)$ and for the lower local amalgamated entropy $h_\nu^l(x, G_1)$ of $\nu$, at every point $x$ in a Borel set $A\subset X$ of full $\nu$-measure. We prove that $h^+_\nu(x, G_1)$ and $h_\nu^l(x, G_1)$ are bounded above by $h(\hat\mu)-h(\pi_{1*}\hat\mu)$, for $x \in A$.

 In \textbf{Section 7}, we study smooth maps $f_j$ which are  injective on a neighbourhood of a $G$-invariant compact set $\Lambda \subset \bb R^D$. We do not assume that the maps $f_j\in G_1, 1 \le j \le m$ have any common stable or unstable tangent subspaces, only that there exist  continuous $G$-invariant unstable and stable \textbf{cone fields} $C_1(\cdot)$ and  $C_2(\cdot)$ respectively on $\Lambda$. Then for any $x \in \Lambda$, $C_1(x)$ contains the unstable tangent spaces $E^u(x, f_j)$,  and $C_2(x)$ contains the stable tangent spaces $E^s(x, f_j)$, for $1 \le j \le m$. The associated \textbf{unstable multi-potential} of $G_1$ on $\Lambda$ is, $$\Phi^u(x) := (-\log m(Df_1| E^u(x, f_1)), \ldots, -\log m(Df_m| E^u(x, f_m)), \ x \in \Lambda,$$ where in general $m(L| E)$ denotes the minimal expansion of a linear map  $L: \mathbb R^D \to \mathbb R^D$ on a linear subspace $E\subset \mathbb R^D$. Then, the amalgamated pressure $P^A$ of $\Phi^u$ is employed to study  the \textbf{Hausdorff dimension of intersections} between $\Lambda$ and arbitrary submanifolds $\Delta \subset B(x, r)$ transversal to the core  subspace of $C_2(x)$, for $x \in \Lambda$. The unique zero $t^{uA}_{G_1}$ of $P^A(t\Phi^u, G_1)$ gives a better dimension estimate, than the estimates obtained with the usual pressure for the individual maps $f_j, $  $1 \le j \le m$.

 \section{Pressure functionals of multi-potentials for finitely generated semigroups.}

Let $(X, d)$ be a compact metric space, a fixed integer $m \ge 1$, and a finite set of continuous maps  $G_1 = \{id_X, f_1, \ldots, f_m\}$ defined on $X$. 
For any  $n \ge 1$, denote by $G_n$ the collection of compositions of at most $n$ functions from $G_1$. The semigroup $G$ is assumed to be generated by $G_1$, with  composition.

Let $\Sigma_m^*$ be the set of all sequences $(i_1, \ldots, i_n)$,  $i_j \in \{1, \ldots, m\}, 1\le j \le n, n \ge 1$. If $C = (i_1, \ldots, i_n)\in \Sigma_m^*$ denote the \textit{length} of $C$ by $n(C)$, and if $j_1, \ldots, j_n \ge 1, n \ge 1$ let $$f_{j_1\ldots j_n} := f_{j_1}\circ \ldots \circ f_{j_n}$$

For any $y \in X$,  $n \ge 1$, $\vp >0$,  $C = (i_1, \ldots, i_n) \in \Sigma_m^*$, define the  \textbf{Bowen ball} at $y$ along the trajectory $C$ and of radius $\vp$ (or the $(n, \vp)$-Bowen ball at $y$ along $C$) by, 
\begin{equation}\label{bowen}
B_n(y, C, G_1, \vp):= \{z \in X, d(y, z) < \vp, d(f_{i_1}(y), f_{i_1}(z)) < \vp, \ldots, d(f_{i_n\ldots i_1}(y), f_{i_n\ldots i_1}(z)) < \vp\}
\end{equation}
Denote by $\Sigma_m^+:= \{\om = (\om_0, \om_1, \ldots), \om_j \in \{1, \ldots, m\}, j \ge 0\}$ the 1-sided shift space on $m$ symbols,  with the canonical metric and shift map $\sigma: \Sigma_m^+ \to \Sigma_m^+$. 
If $G_1$ is given, and $\om \in \Sigma_m^+$, we also denote the above $(n, \vp)$-Bowen ball along the trajectory $\om$ by, 
\begin{equation}\label{bowensimple}
B_n(x, \om, \vp):=\{z\in X, d(x, z) < \vp, d(f_{\om_0}(x), f_{\om_0}(z) < \vp, \ldots, d(f_{\om_{n-1}\ldots \om_0}(x), f_{\om_{n-1}\ldots \om_0}(z)) < \vp\}\end{equation}
If $\om \in \Sigma_m^+$ denote $\om|_n$ the truncation to the first $n$ elements of $\om$, \ $\om|_n = (\om_0, \ldots, \om_{n-1})$.

For  $m \ge 1$ from above,  consider a continuous \textbf{multi-potential} on $X$, namely $$\Phi := (\phi_1, \ldots, \phi_m): X \to \mathbb R^m,$$  where $\phi_1, \ldots, \phi_m$ are  real-valued continuous potentials on $X$.

 If $(C, x) \in \Sigma_m^* \times X$, $C = (i_1, \ldots, i_n)$,  define the  \textbf{consecutive sum of $\Phi$ on $(C, x)$}, 
\begin{equation}\label{consum}
S_n\Phi(x, C):= \phi_{i_1}(x) + \phi_{i_2} (f_{i_1}(x)) + \ldots + \phi_{i_n}(f_{i_{n-1}\ldots i_1}(x))
\end{equation}
If $\om \in \Sigma_m^+$, then for any $n \ge 1$, denote the above consecutive sum also by $$S_n\Phi(x, \om) := S_n\Phi(x, (\om_0, \ldots \om_{n-1}))$$
Given now $m \ge 1$ and the above finite generator set $G_1$,  denote by $\mathcal X_m := \Sigma_m^* \times X$.

We  define now several types of pressure for $\Phi$,   using  Carath\'eodory-Pesin structures (\cite{Pe}):

\ \ \ \textbf{1. Amalgamated Pressure.}

Consider  a subset $Y \subset X$, and cover $Y$ with Bowen balls of type $B_n(y, C, G_1, \vp)$, for various points $y \in X$, integers $n$ and trajectories $C$.  
  For arbitrary $\lambda \in \mathbb R$, $\Phi \in \mathcal C(X, \mathbb R^m)$, $N$ a positive integer, and arbitrary $\vp >0$, define:
  \begin{equation}\label{M-}
  \begin{aligned}
   M(\lambda, \Phi, Y, G_1, N, \vp) &:= 
  \inf\{\mathop{\sum}\limits_{(C, y) \in \Gamma} \exp (-\lambda n(C) + S_{n(C)}\Phi(y, C)),  \ \Gamma \subset \mathcal X_m, \\ &Y \subset \mathop{\bigcup}\limits_{(C, y) \in \Gamma} B_{n(C)}(y, C, G_1, \vp), \ \text{and} 
   \ n(C) \ge N,   \forall  (C, y) \in \Gamma\}
  \end{aligned}
  \end{equation}
  When $Y = X$, write simply $M(\lambda, \Phi, G_1, N, \vp)$.
  If $N$ increases,  the set of collections $\Gamma$ becomes smaller, so the infimum over $\Gamma$ increases. Hence the following limit exists:
  \begin{equation}\label{ml}
  m(\lambda, \Phi, Y, G_1, \vp) := \mathop{\lim}\limits_{N \to \infty}M(\lambda, \Phi, Y, G_1, N, \vp)
  \end{equation}
  
 But  $X$ is compact and every map  $\phi_i, i = 1, \ldots, m$ is bounded on $X$, so for every $N \ge 1$ there must exist at least one cover of $X$ with Bowen balls of type $B_n(y, C, G_1, \vp)$ with $n > N$, and there exists $M>0$ so that, for any $(C, y) \in \mathcal X_m$, $|S_{n(C)}\Phi(y, C)| \le n(C)M$. Thus there exists $\lambda \in \mathbb R$ large so that $m(\lambda, \Phi, Y, G_1, \vp) = 0$.
Then denote,
 \begin{equation}\label{P-vp}
 P^A(\Phi, Y, G_1, \vp):= \inf\{\lambda, m(\lambda, \Phi, Y, G_1, \vp) = 0\}
 \end{equation}
 If $\vp>0$ decreases, then $m(\lambda, \Phi, Y, G_1, \vp)$ increases and $P^A(\Phi, Y, G_1, \vp)$ increases; so there exists $\mathop{\lim}\limits_{\vp \to 0} P^A(\Phi, Y, G_1, \vp)$. We thus define the following functional:
  
  \begin{definition}\label{invpres}
  For any set $Y \subset X$ and any multi-potential  $\Phi = (\phi_1, \ldots, \phi_m)\in \cal C(X, \bb R^m)$ define the  \textbf{amalgamated topological pressure} of $\Phi$ on $Y$, with respect to the semigroup $G$ of maps on $X$ generated by the finite set $G_1$, by $$P^A(\Phi, Y, G_1) := \mathop{\lim}\limits_{\vp \to 0} P^A(\Phi, Y, G_1, \vp)$$  
  The \textbf{amalgamated topological entropy} of $G_1$ on $Y$ is $h^A(Y, G_1) := P^A(\textbf{0}, Y, G_1)$.
 \newline 
  When $Y = X$,  write $P^A(\Phi, G_1)$.  If $G_1 = \{id_X, f\}$ and $\Phi = \phi \in \mathcal C(X, \mathbb R)$, then we obtain the usual topological pressure  $P(\phi, f)$.
 \end{definition}
 
 Let us form also the expressions,
 \begin{equation}\label{cn}
 C_n(\Phi, Y, G_1, \lambda, \vp) := \inf\{\mathop{\sum}\limits_{(x, \omega) \in \mathcal F} \exp(S_n\Phi(x, \omega) - n\lambda),  \  \mathcal F \subset X \times \Sigma_m^+,  Y \subset \mathop{\cup}\limits_{(x, \omega) \in \mathcal F} B_n(x, \omega, \vp)\}
 \end{equation}
 
 Then take $\u C(\Phi, Y, G_1, \lambda, \vp) := \liminf_{n\to \infty} C_n(\Phi, Y, G_1, \lambda, \vp)$, and $\o  C(\Phi, Y, G_1, \lambda, \vp) := \limsup_{n\to \infty} C_n(\Phi, Y, G_1, \lambda, \vp)$.
 We now define $$\u{CP}^A(\Phi, Y, G_1, \vp) := \inf\{\lambda, \u C(\Phi, Y, G_1, \lambda, \vp) = 0\} = \sup\{\lambda, \u C(\Phi, Y, G_1, \lambda, \vp) = \infty\}$$
 $$\o{CP}^A(\Phi, Y, G_1, \vp) :=  \inf\{\lambda, \o C(\Phi, Y, G_1, \lambda, \vp) = 0\} = \sup\{\lambda, \o C(\Phi, Y, G_1, \lambda, \vp) = \infty\}.$$

 Then, define the \textbf{lower amalgamated capacity pressure} and the \textbf{upper amalgamated capacity pressure} of the multi-potential $\Phi$ on the set $Y$ respectively, by:
 \begin{equation}\label{acp}
 \u{CP}^A(\Phi, Y, G_1) := \mathop{\lim}\limits_{\vp \to 0} \u{CP}^A(\Phi, Y, G_1, \vp),
 \ \o{CP}^A(\Phi,  Y, G_1) := \mathop{\lim}\limits_{\vp \to 0} \o{CP}^A(\Phi, Y, G_1, \vp).
\end{equation}

 We define now a notion of \textbf{amalgamated pressure along a set of trajectories}.
In the above setting, let a set of trajectories $\cal Y \subset \Sigma_m^+ \times X$, and a multi-potential $\Phi \in \mathcal C(X, \bb R^m)$.   
Then for any $N > 1$ and $\vp>0$, consider the expression:
$$
\begin{aligned}
 M(\lambda, \Phi, \cal Y &,  G_1, N,   \vp) := 
  \inf\{\mathop{\sum}\limits_{(\om, y) \in \Gamma} \exp (-\lambda n + S_{n}\Phi(y, \om)),  \ \Gamma \subset \cal Y, \  \Gamma \ \text{countable}, \\ 
  &\text{such that} \  \pi_2(\cal Y) \subset \mathop{\bigcup}\limits_{(\om, y) \in \Gamma} B_{n}(y, \om, \vp) \ \text{and}  \ n \ge N, \forall \ (\om, y) \in \Gamma\}
  \end{aligned}
$$
 As before, the following limit exists:
 $
 m(\lambda, \Phi, \cal Y, G_1, \vp) := \mathop{\lim}\limits_{N \to \infty} M(\lambda, \Phi, \cal Y, G_1, \vp)$.
 Next, define $P^A(\Phi, \cal Y, G_1, \vp) := \inf\{\lambda, m(\lambda, \Phi, \cal Y, G_1, \vp) = 0\}$.
 
 \begin{definition}\label{calydef}
 The \textbf{amalgamated pressure of $\Phi$ along trajectory set   $\cal Y$} is
 \begin{equation}\label{caly}
 P^A(\Phi, \cal Y, G_1) := \mathop{\lim}\limits_{\vp \to 0} P^A(\Phi, \cal Y, G_1, \vp).
 \end{equation}
Define  the \textbf{amalgamated entropy along $\cal Y$}, by $h^A(\cal Y, G_1):= P^A(\textbf{0}, \cal Y, G_1)$. When we want to emphasize also the projection set we write $h^A(\pi_2(\cal Y), \cal Y, G_1)$ for $h^A(\cal Y, G_1)$.
\end{definition}

 In general, the amalgamated pressure along a set of trajectories  \textbf{is not}  a Carath\'eodory-Pesin structure, since it does not satisfy the monotonicity property, see Remark \ref{calyne}.
 
The \textbf{upper capacity} $\o{CP}^A(\Phi, \cal Y, G_1)$, and the \textbf{lower capacity} $\u{CP}^A(\Phi, \cal Y, G_1)$, of the multi-potential $\Phi$ along the set of trajectories $\mathcal Y\subset \Sigma_m^+\times X$, are defined as before.

 \
 
 \ \ \textbf{2. Condensed Pressure.}

 Define the $n$-\textbf{condensed Bowen ball} centered at $x \in X$ of radius $\vp>0$, with respect to the semigroup $G$ generated by $G_1$ finite, by
\begin{equation}\label{cond}
B_n(x,  G_1, \vp) := \mathop{\bigcap}\limits_{\om \in \Sigma_m^+} B_n(x, \om, \vp)
\end{equation}
 Note that since $G_1$ is finite, then $B_n(x, G_1, \vp)$ is a neighbourhood of $x$.
 
Define the \textbf{lower/upper consecutive sum} of the multi-potential $\Phi$ at $x\in X$ by, 
 \begin{equation}\label{lcs}
 s_n\Phi(x) := \inf\{S_n\Phi(x, \om), \om \in \Sigma_m^+\}, \  \ 
 S_n\Phi(x) := \sup\{S_n\Phi(x, \om), \om \in \Sigma_m^+ \},
 \end{equation}
 and when we want to emphasize dependence on $G_1$ we write also $s_n\Phi(x, G_1)$, respectively $S_n\Phi(x, G_1)$.
 Then, for a subset $Y \subset X$ and any integer $N\ge 1$ define, $$M_{l, N}(\Phi, Y, G_1, \lambda,  \vp) := \inf\{\mathop{\sum}\limits_{x \in \mathcal F} \exp(s_n\Phi(x) - \lambda n), Y \subset \mathop{\cup}\limits_{x\in \mathcal F} B_n(x, G_1, \vp), n \ge N\},$$
 $$M_{u, N}(\Phi, Y, G_1, \lambda, \vp) :=  \inf\{\mathop{\sum}\limits_{x \in \mathcal F} \exp(S_n\Phi(x) - \lambda n), Y \subset \mathop{\cup}\limits_{x\in \mathcal F} B_n(x, G_1, \vp), n \ge N\}$$

When $N$ increases, the collection of coverings of $Y$ decreases, thus the infimum in the expressions above increases, hence there exist the limits
$$m_l(\Phi, Y, G_1, \lambda, \vp) := \mathop{\lim}\limits_{N \to \infty}M_{l, N}(\Phi, Y, G_1, \lambda, \vp), \  \ 
m_u(\Phi, Y, G_1, \lambda, \vp) := \mathop{\lim}\limits_{N \to \infty} M_{u, N}(\Phi, Y, G_1, \lambda, \vp)$$

 Then as usual with a Carath\'eodory-Pesin structure (see \cite{Pe}, \cite{PP}), we take:
 $$P_l(\Phi, Y, G_1, \vp) := \inf\{\lambda, m_l(\Phi, Y, G_1, \vp) = 0 \} = \sup\{\lambda, m_l(\Phi, Y, G_1, \lambda, \vp) = \infty \},$$
 $$P_u(\Phi, Y, G_1, \vp) := \inf\{\lambda, m_u(\Phi, Y, G_1, \vp) = 0\} = \sup \{\lambda, m_u(\Phi, Y, G_1, \lambda, \vp) = \infty\}$$
 
 Since the generating set $G_1$ is finite, and the potentials $\phi_j$ are continuous on $X$, we know that their maximum oscillation converges to 0, when $\vp \to 0$. Hence the following limits exist, and we define the \textbf{lower condensed pressure} and the \textbf{upper condensed presure} of the multi-potential $\Phi$ on $Y$ respectively, by
 \begin{equation}\label{condp}
 P_l(\Phi, Y, G_1) := \mathop{\lim}\limits_{\vp \to 0} P_l(\Phi, Y, G_1, \vp), \ 
 P_u(\Phi, Y, G_1) := \mathop{\lim}\limits_{\vp \to 0} P_u(\Phi, Y, G_1, \vp)
\end{equation} 
 
 We notice that a notion of topological entropy for semigroups, defined with  sets of type $B_n(x, G_1, \vp)$ was studied in  \cite{Bi} and \cite{HS}.
 
By covering with balls along trajectories of the same length,  define now as above the \textbf{lower/upper capacity pressures} for the lower condensed pressure, respectively $\u{CP}_l(\Phi, Y, G_1)$ and $ \o{CP}_l(\Phi, Y, G_1)$, and the \textbf{lower/upper capacity for the upper condensed pressure}, namely $\u{CP}_u(\Phi, Y, G_1)$ and $\o{CP}_u(\Phi, Y, G_1)$.
 %

 \
 
 \ \ \textbf{3. Trajectory Pressure.}

 Trajectory pressure is obtained  if we fix  $\om \in \Sigma_m^+, \om = (\om_0, \om_1, \ldots)$, and iterate only along $\om$. It is a particular case of the pressure along a set of trajectories in (\ref{caly}).
Indeed if $Y \subset X$ and $\om \in \Sigma_m^+$, then the \textbf{trajectory pressure on $Y$ along  $\om$} denoted by $P(\Phi, Y, G_1, \om)$, is equal to $P^A(\Phi, \cal Y, G_1)$, with $\cal Y = \{\om\} \times Y$. 

  
 \
 
 \ \ \textbf{4. Exhaustive Pressure.}
 
 Another type of pressure is obtained by covering $Y \subset X$ with sets of type $$B_n^+(x, G_1, \vp) := \mathop{\bigcup}\limits_{|\om| = n} B_n(x, \om, \vp),$$
 called the \textbf{exhaustive $n$-Bowen ball} centered at $x$ and of radius $\vp$. So $B_n^+(x, G_1, \vp)$ is the union of $(n, \vp)$-Bowen balls at $x$ over all possible trajectories of length $n$. Recall from (\ref{lcs}) that $S_n\Phi(x) := \sup\{S_n\Phi(x, \om), \om \in \Sigma_m^+\}, x \in X, n \ge 1$.

 As before if $Y \subset X$, $N>0$, $\lambda \in \bb R$, we form the expression,
 $$M_u^+(\Phi, Y, G_1, \lambda, \vp, N):= \inf\{\mathop{\sum}\limits_{x\in \mathcal F} \exp(S_n\Phi(x) - \lambda n), \ Y \subset \mathop{\cup}\limits_{x \in \mathcal F} B_n^+(x, G_1, \vp), n \ge N\}.$$
 
 Similarly, recall from (\ref{lcs}) that $s_n\Phi(x) := \inf\{S_n\Phi(x, \om), \om \in \Sigma_m^+\}$, and define:
 $$M_l^+(\Phi, Y, G_1, \lambda, N, \vp) := \inf\{\mathop{\sum}\limits_{x\in \mathcal F} \exp(s_n\Phi(x, G_1) - \lambda n), \ Y \subset \mathop{\cup}\limits_{x \in \mathcal F} B_n^+(x, G_1, \vp), n \ge N\}.$$

 Then we obtain $m_u^+(\Phi, Y, G_1, \lambda, \vp) := \mathop{\lim}\limits_{N \to \infty} M^+_u(\Phi, Y, G_1, \lambda, N, \vp)$, and $m_l^+(\Phi, Y, G_1, \lambda, \vp) :=  \mathop{\lim}\limits_{N \to \infty} M^+_l(\Phi, Y, G_1, \lambda, N, \vp)$.
 This defines the following: $$P^+_u(\Phi, Y, G_1, \vp) = \inf\{\lambda, m^+_u(\Phi, Y, G_1, \lambda, \vp) = 0\},  \ P_l^+(\Phi, Y, G_1, \vp) = \inf\{\lambda, m^+_u(\Phi, Y, G_1, \lambda, \vp) = 0\}$$
 
From above the following limits exist when $\vp \to 0$ and  define the \textbf{upper exhaustive pressure}, and  \textbf{lower exhaustive pressure} of $\Phi$ on $Y$ respectively, 
 \begin{equation}\label{exh}
 P^+_u(\Phi, Y, G_1) := \mathop{\lim}\limits_{\vp \to 0} P^+_u(\Phi, Y, G_1, \vp), \ 
 P^+_l(\Phi, Y, G_1) := \mathop{\lim}\limits_{\vp \to 0} P^+_l(\Phi, Y, G_1, \vp).
 \end{equation}
 
\
 
 \section{Properties and comparisons of pressures.}
 
 In this Section we will explore properties of the various types of pressure for multi-potentials from Section 2, and will compare between them. Moreover, we will provide examples showing that in general these notions of pressure are different.

 \begin{theorem}\label{amal}
 
 Consider a compact metric space $X$, a semigroup $G$ of continuous self-maps of $X$ generated by the finite set $G_1$, and let $\Phi \in \mathcal C(X, \bb R^m)$ be a continuous multi-potential on $X$, and $Y$ be a subset of $X$. Then the amalgamated pressure satisfies the following properties:
 
\  a) If $Z_1 \subset Z_2$, then $P^A(\Phi, Z_1, G_1) \le P^A(\Phi, Z_2, G_1)$.

\ b) Let $Z_j \subset X, j \ge 1$ and $Z:= \mathop{\cup}\limits_{j \ge 1} Z_j$.  Then $P^A(\Phi, Z, G_1) = \mathop{\sup}\limits_{j \ge 1} P^A(\Phi, Z_j, G_1)$.

\ c) If $g:X \to Z$ is a homeomorphism of two compact metric spaces, and $Y \subset X$, then $P^A(\Phi, Y, G_1) = P^A(\Phi\circ g^{-1}, g(Y), g\circ G_1\circ g^{-1})$, where $\Phi\circ g^{-1} = (\phi_1\circ g^{-1}, \ldots, \phi_m\circ g^{-1})$ and $g\circ G_1 \circ g^{-1} = \{g\circ \kappa \circ g^{-1}, \kappa \in G_1\}$.
 
 \ d) If $Z_1 \subset Z_2$, then $\u{CP}^A(\Phi, Z_1, G_1) \le \u{CP}^A(\Phi, Z_2, G_1)$, and $\o{CP}^A(\Phi, Z_1, G_1) \le \o{CP}^A(\Phi, Z_2, G_1)$.
 
 \ e) If $Z_j, j \ge 1$ are subsets of $X$ and $Z:= \mathop{\cup}\limits_{j \ge 1} Z_j$, then $\u{CP}^A(\Phi, Z, G_1) \ge \sup_{j\ge 1} \u{CP}^A(\Phi, Z_j, G_1)$, and $\o{CP}^A(\Phi, Z, G_1) \ge \sup_{j \ge 1} \o{CP}^A(\Phi, Z_j, G_1)$. 
 
 \ f) If $g:X \to X$ is a homeomorphism which commutes with every $f_i$ from $G_1$, then $\u{CP}^A(\Phi, g(Y), G_1) = \u{CP}^A(\Phi\circ g, Y, G_1)$, and $\o{CP}^A(\Phi, g(Y), G_1) = \o{CP}^A(\Phi\circ g, Y, G_1)$.
 
 \ g)  $P^A(\Phi, Y, G_1) \le \u{CP}^A(\Phi, Y, G_1) \le \o{CP}^A(\Phi, Y, G_1)$.
 \end{theorem}

 \begin{proof}
 The proof for a), c), d), f) is straightforward. For b), it is clear from a) that $P^A(\Phi, Z, G_1) \ge P^A(\Phi, Z_j, G_1)$ for all $j \ge 1$. Let now $\lambda > \sup_jP^A(\Phi, Z_j, G_1)$. Then $m(\Phi, Z_j, G_1, \lambda, \vp) = 0, j \ge 1$. Hence for all $N \ge 1, \vp>0$ and all $j \ge 1$, $M_N(\Phi, Z_j, G_1, \lambda, \vp) = 0$. 
 This implies that for all $j \ge 1$ there exists a family $\mathcal F_j \in \mathcal X_m$ such that $Z_j \subset \mathop{\cup}\limits_{(C, x) \in \mathcal F_j} B_{n(C)}(x, C, \vp)$ and for any $(C, x) \in \mathcal F_j$ we have $n(C) >N$ and, $$\mathop{\sum}\limits_{(C, x) \in \mathcal F_j} \exp(S_{n(C)}\Phi(x, C) - n(C)\lambda) < \frac{\vp}{2^j}$$
 Consider  $\mathcal F := \mathop{\bigcup}\limits_{j \ge 1} \mathcal F_j$. Then $Z \subset \mathop{\bigcup}\limits_{(C, x) \in \mathcal F} B_{n(C)}(x, C, \vp)$, and 
 $
 \mathop{\sum}\limits_{(C, x) \in \mathcal F} \exp(S_{n(C)}\Phi(x, C) - n(C)\lambda) < \mathop{\sum}\limits_{j \ge 1} \frac{\vp}{2^j} = \vp$.
Thus $P^A(\Phi, Z, G_1) \le \lambda$, and $P^A(\Phi, Z, G_1) = \sup_{j \ge 1} P^A(\Phi, Z_j, G_1)$.

For g) notice that for any $N$ fixed, the capacity pressures use only balls of type $B_N(x, \om, \vp)$, while $P^A$ uses balls $B_{n_i}(y, \eta, \vp)$,  $n_i \ge N$. 
\end{proof}

\begin{remark}\label{calyne}
Regarding the properties of the amalgamated pressure over sets of trajectories,  if $Z \subset X$ and $A_1 \subset A_2 \subset \Sigma_m^+$, then $$P^A(\Phi, Z\times A_1, G_1) \ge P^A(\Phi, Z \times A_2, G_1)$$
And if $Z_1 \subset Z_2 \subset X$ and $A \subset \Sigma_m^+$, then $P^A(\Phi, Z_1\times A, G_1) \le P^A(\Phi, Z_2\times A, G_1)$.

However if $\cal Y_1 \subset \cal Y_2 \subset \Sigma_m^+\times X$, then the quantities $P^A(\Phi, \cal Y_1, G_1)$ and $P^A(\Phi, \cal Y_2, G_1)$ are incomparable in general. This is because, even if $\pi_2(\cal Y_1) \subset \pi_2(\cal Y_2) \subset X$, in definition of $P^A$ 
we take infimum over a smaller set $\pi_1(\cal Y_1) \subset \pi_1(\cal Y_2)$. 
Thus  if $\cal Y \subset \Sigma_m^+ \times X$, it does not follow in general that $P^A(\Phi, G_1) \ge P^A(\Phi, \cal Y, G_1)$. 
\end{remark}

 \begin{theorem}\label{oscpressure}

Consider continuous multi-potentials $\Phi, \Psi \in \mathcal C(X, \bb R^m)$, and $Y$ a subset of a compact metric space $X$. Then,

\ a) $|P^A(\Phi, Y, G_1) - P^A(\Psi, Y, G_1)| \le ||\Phi - \Psi||:= \mathop{\sup}\limits_{1 \le j \le m}||\phi_j - \psi_j||.$

\ b) $|\u{CP}^A(\Phi, Y, G_1) - \u{CP}^A(\Psi, Y, G_1)| \le \mathop{\sup}\limits_{1 \le j \le m}||\phi_j - \psi_j||$.

\ c) $|\o{CP}^A(\Phi, Y, G_1) - \o{CP}^A(\Psi, Y, G_1) | \le \mathop{\sup}\limits_{1 \le j \le m}||\phi_j - \psi_j||.$
\end{theorem}

 \begin{proof}
 For a), notice that if $\mathcal F$ is a family of balls that cover $Y$, then for any $(C, x) \in \mathcal F$,  $|S_{n(C)}\Phi(x, C) - S_{n(C)}\Psi(x, C)| \le n(C) ||\Phi - \Psi||$. From definition of  $P^A$ we obtain the inequality. 
 The other two inequalities for the capacity pressure follow similarly.
 \end{proof}
 
 Let us also remark that the same notions of amalgamated pressure and lower/upper capacity amalgamated pressure are obtained if we cover with arbitrary open sets. Namely, let $\mathcal U $ be a finite open cover of the compact metric space $X$, and for $C = (j_1, \ldots, j_{n}) \in \Sigma_m^*$ and a string $\bf U = (U_{i_0}, U_{i_1}, \ldots, U_{i_{n-1}}) \in \mathcal U^n$ denote by  
 \begin{equation}\label{xu}
 X(\bf U, C) = \{x \in X, x \in U_{i_0}, f_{i_1}(x) \in U_{i_1}, \ldots, f_{i_n\ldots i_1}(x) \in U_{i_n}\}
 \end{equation}
 Then we can define as in \cite{Pe} the expressions of type
 \begin{equation}\label{mxu}
 M_N(\Phi, Y, G_1, \lambda, \cal U) := \inf\{\mathop{\sum}\limits_{(\bf U, C) \in \cal G}\exp(\mathop{\sup}\limits_{y \in X(\bf U, C)} S_{n(C)}\Phi(y, C) - \lambda n)\},
 \end{equation}
 where the infimum is taken over all collections $\mathcal G$  which cover $Y$ with sets of type $X(\bf U, C)$, with $n(C) \ge N$. Using this, define as before $P^A(\Phi, Y, G_1, \mathcal U)$. Denote also 
 \begin{equation}\label{supS}
 S_{n(C)}(\bf U, C) := \mathop{\sup}\limits_{y \in X(\bf U, C)} S_{n(C)}\Phi(y, C)
 \end{equation}
 Define as before $\u{CP}^A(\Phi, Y, G_1, \mathcal U)$ and $\o{CP}^A(\Phi, Y, G_1, \mathcal U)$.
 As above we obtain:
 
 \begin{proposition}\label{proxu}
 In the above setting, 
 $P^A(\Phi, Y, G_1) = \mathop{\lim}\limits_{|\mathcal U| \to 0} P^A(\Phi, Y, G_1, \mathcal U),$
 $$\u{CP}^A(\Phi, Y, G_1) = \mathop{\lim}\limits_{|\cal U| \to 0}\u{CP}^A(\Phi, Y, G_1, \mathcal U), \ 
 \o{CP}^A(\Phi, Y, G_1) = \mathop{\lim}\limits_{|\cal U| \to 0}\o{CP}^A(\Phi, Y, G_1, \mathcal U).$$
  \end{proposition}

 \begin{definition}\label{iset}
 A subset $Y \subset X$ is called $G_1$-invariant if $Y$ is an $f_j$-invariant set for all $1 \le j \le m$.
 \end{definition}

 \begin{theorem}\label{invset}
 If $Y$ is a $G_1$-invariant set, then $\u{CP}^A(\Phi, Y, G_1) = \o{CP}^A(\Phi, Y, G_1)$. 
 
Also if $Y$ is compact $G_1$-invariant, then $P^A(\Phi, Y, G_1) = \u{CP}^A(\Phi, Y, G_1) = \o{CP}^A(\Phi, Y, G_1)$. 
\end{theorem}  
 
 \begin{proof}
 Let  $k, n \ge 1$, and assume $k \le n$. Let us take a family $\mathcal F_n$ of balls of type $B_n(x, \om, \vp)$ which cover $Y$, and a family $\mathcal F_k$ of balls $B_k(y, \eta, \vp)$ which cover $Y$, where $(\om, x), (\eta, y) \in \mathcal X_m$.  Now each iterate $f_{\om_{n-1}\ldots \om_0}(x)$ belongs to an $k$-ball $B_m(y, \eta, \vp)$ for some $(\eta, y) \in \mathcal F_k$, as $Y$ is $G_1$-invariant. We can thus concatenate the $n$- and then the $k$-trajectories to form $(n+k)$-trajectories. This concatenation is the key here.
Thus,   we can cover $Y$ with $(n+k)$-balls of type $B(x, \om\eta, 2\vp)$. Moreover,
$$
S_{n+k}\Phi(x, \om\eta) \le  \phi_{\om_0}(x) + \ldots + \phi_{\om_n}(f_{\om_{n-1}\ldots \om_0}(x)) +
\phi_{\eta_0}(y) + \ldots + \phi_{\eta_k}(f_{\eta_{k-1}\ldots\eta_0}(y)) + k o(\vp),
$$ 
 if $f_{\om_{n-1}\ldots\om_0}(x) \in B_k(y, \eta, \vp)$, where $o(\vp)$ is the maximum oscillation of  $\phi_j, 1 \le j \le m$ on a ball of radius $\vp$. 
Recalling the definition (\ref{cn}), we obtain
\begin{equation}\label{SO}
 C_{n+k}(\Phi, Y, G_1, 2\vp) \le C_n(\Phi, Y, G_1, \vp) \cdot C_m(\Phi, Y, G_1, \vp) \cdot e^{ko(\vp)}
 \end{equation}
Then, for any integer $p = \ell n + k, k \le  n$, 
 $\frac{C_p(\Phi, Y, G_1, 2\vp)}{p} \le \inf\frac{C_n(\Phi, Y, G_1, \vp)}{n} + o(\vp)$, so $$\limsup_p \frac{C_p(\Phi, Y, G_1, 2\vp)}{p} \le \inf_n \frac{C_n(\Phi, Y, G_1, 2\vp)}{n} + o(\vp),  \text{and} \ \liminf_p  \frac{C_p(\Phi, Y, G_1, 2\vp)}{p} \ge \inf_n  \frac{C_n(\Phi, Y, G_1, 2\vp)}{n}$$
 Therefore by letting $\vp \to 0$, we obtain $\u{CP}^A(\Phi, Y, G_1) = \o{CP}^A(\Phi, Y, G_1)$.

 For the second part,  assume $Y$ is now compact and $G_1$-invariant. From Theorem \ref{amal} g) and the first part, we know $P^A(\Phi, Y, G_1) \le \u{CP}^A(\Phi, Y, G_1) = \o{CP}^A(\Phi, Y, G_1)$.
 So consider $\alpha > P^A(\Phi, Y, G_1)$ arbitrary. Recall also (\ref{supS}).  Then, from Proposition \ref{proxu} and as $Y$ is compact, it follows that for any $N \ge 1$ there exists an open cover $\cal U$ and a finite collection $\cal G$ of sets of type $X(\bf U, C)$ which cover  $Y$, and  $\exists \beta \in (0, 1)$ such that 
 \begin{equation}\label{betag}
  \mathop{\sum}\limits_{(\bf U, C) \in \cal G} \exp(S_{n(C)}(\bf U, C) - \alpha n(C)) < \beta < 1,
  \end{equation}
  and $n(C) \ge N$ for all $(\bf U, C) \in \cal G$. 
 But now we use again the concatenation of strings $\bf U$, and for any integer $n \ge 1$ form the collection $\cal G^n$ of sets  $X(\bf U^1\ldots \bf U^n, C_1\ldots C_n)$,  for $(U^j, C_j) \in \cal G, 1 \le j \le n$. 
Also consider the union  $\hat{\cal G} := \mathop{\bigcup}\limits_{n \ge 1} {\cal G}^n$. From (\ref{betag}), $\mathop{\sum}\limits_{(\bf V, C') \in \mathcal G^n} \exp(S_{n(C')}(\bf V, C') - \alpha n(C')) \le \beta^n$. 
So there is $\gamma < \infty$ such that 
\begin{equation}\label{gammas}
\mathop{\sum}\limits_{(\bf V, C') \in \hat{\cal G}} \exp(S_{n(C')}(\bf V, C') - \alpha n(C')) \le \mathop{\sum}\limits_{n \ge 1} \beta^n < \gamma < \infty
\end{equation}
Let $M$ be the maximum length of a sequence from  $\mathcal G$. For any  $n\ge 1$, we can reach  length $n$ by sequences in $\hat{\cal G}$. Thus we cover $Y$ using sequences $C'$ of length between $n$ and $n+M$. 
If $C' = (i_1, \ldots, i_{n+M})$, then $S_{n+M}\Phi(\bf V, C') \le S_n\Phi((V_1, \ldots, V_n), (i_1, \ldots, i_n)) + M ||\Phi||$.
Thus from (\ref{gammas}),  $\alpha > \o{CP}^A(\Phi, Y, G_1, \mathcal U)$.
But $\alpha$ is arbitrary and $\cal U$ arbitrary, so $P^A(\Phi, Y, G_1) \ge \o{CP}^A(\Phi, Y, G_1)$. So $P^A(\Phi, Y, G_1) = \u{CP}^A(\Phi, Y, G_1) = \o{CP}^A(\Phi, Y, G_1).$

 \end{proof}
 
Now, define the skew-product map $F: \Sigma_m^+ \times X \to \Sigma_m^+ \times X$,
\begin{equation}\label{FF}
F(\om, x) = (\sigma(\om), f_{\om_0}(x)), \ \text{for any} \  (\om, x) \in \Sigma_m^+ \times X
\end{equation}
 This skew-product is very useful in dynamics of semigroups, random systems (for eg \cite{Bu}). Similarly as above,
 
 \begin{theorem}\label{invcaly}
 \ a) If $\cal Y\subset \Sigma_m^+\times X$ is $F$-invariant, then $\u{CP}^A(\Phi, \cal Y, G_1) = \o{CP}^A(\Phi, \cal Y, G_1)$. 
 
\ b)  If $\cal Y$ is $F$-invariant in $\Sigma_m^+\times X$ and its canonical projection on  second coordinate is compact in $X$, then $P^A(\Phi, \cal Y, G_1) = \u{CP}^A(\Phi, \cal Y, G_1) = \o{CP}^A(\Phi, \cal Y, G_1)$. 
\end{theorem}

 Similarly to the case of the amalgamated pressure, one can prove:
 
 \begin{theorem}\label{propcond}
 Let $X$ be a compact metric space and $\Phi \in \mathcal C(X, \bb R^m)$ a multi-potential on $X$. Let also $G_1 = \{id_X, f_1, \ldots, f_m\}$ a finite set that generates the semigroup $G$ of maps of $X$. Then the following properties hold:
 
 \ a) If $Z_1 \subset Z_2$, then $P_l(\Phi, Z_1, G_1) \le P_l(\Phi, Z_2, G_1)$, and $P_u(\Phi,  Z_1, G_1) \le P_u(\Phi, Z_2, G_1)$. 
 
 \ b) Let $Z_i \subset X, 1 \le i$ and $Z := \mathop{\cup}\limits_{i \ge 1} Z_i$. Then $P_l(\Phi, Z, G_1) = \mathop{\sup}\limits_{i \ge 1} P_l(\Phi, Z_i, G_1)$, and $P_u(\Phi, Z, G_1) = \mathop{\sup}\limits_{i \ge 1} P_u(\Phi, Z_i, G_1)$.
 
 \ c) If $g:X \to Z$ is a homeomorphism of two compact metric spaces, and $Y \subset X$, then $P_l(\Phi, Y, G_1) = P_l(\Phi\circ g^{-1}, g(Y), g\circ G_1\circ g^{-1})$, and $P_u(\Phi, Y, G_1) = P_u(\Phi\circ g^{-1}, g(Y), g\circ G_1\circ g^{-1})$, where $\Phi\circ g^{-1} = (\phi_1\circ g^{-1}, \ldots, \phi_m\circ g^{-1})$ and $g\circ G_1 \circ g^{-1} = \{g\circ \kappa \circ g^{-1}, \kappa \in G_1\}$.
 
 \ d) If $Z_1 \subset Z_2$, then $\u{CP}_l(\Phi, Z_1, G_1) \le \u{CP}_l(\Phi, Z_2, G_1)$, and $\o{CP}_l(\Phi, Z_1, G_1) \le \o{CP}_l(\Phi, Z_2, G_1)$. Similarly for $\u{CP}_u$ and $\o{CP}_u$.
 
 \ e) If $Z_j, j \ge 1$ are subsets of $X$ and $Z:= \mathop{\cup}\limits_{j \ge 1} Z_j$, then $\u{CP}_l(\Phi, Z, G_1) \ge \sup_{j \ge 1} \u{CP}_l(\Phi, Z_j, G_1)$, and $\o{CP}_l(\Phi, Z, G_1) \ge \sup_{j \ge 1} \o{CP}_l(\Phi, Z_j, G_1)$. Similarly for   $\u{CP}_u$ and $\o{CP}_u$.
 
 \ f) If $g:X \to X$ is a homeomorphism which commutes with every  map $f_i$ from $G_1$, then $\u{CP}_l(\Phi, g(Y), G_1) = \u{CP}_l(\Phi\circ g, Y, G_1)$, and $\o{CP}_l(\Phi, g(Y), G_1) = \o{CP}_l(\Phi \circ g, Y, G_1)$. Similarly for   $\u{CP}_u$ and $\o{CP}_u$.
 
 \ g)  $P_l(\Phi, Y, G_1) \le \u{CP}_l(\Phi, Y, G_1) \le \o{CP}_l(\Phi, Y, G_1)$, and  $P_u(\Phi, Y, G_1) \le \u{CP}_u(\Phi, Y, G_1) \le \o{CP}_u(\Phi, Y, G_1)$.
 \end{theorem}

 \begin{theorem}\label{osccond}
 If $Y \subset X$ and $\Phi, \Psi \in \mathcal C(X, \bb R^m)$, then 
 $$|P_l(\Phi, Y, G_1) - P_l(\Psi, Y, G_1)| \le ||\Phi - \Psi||, \ \text{and} \ 
 |P_u(\Phi, Y, G_1) - P_u(\Psi, Y, G_1)| \le ||\Phi - \Psi||,$$
 $$|\u{CP}_l(\Phi, Y, G_1) - \u{CP}_l(\Psi, Y, G_1)| \le ||\Phi - \Psi||, \ 
 |\o{CP}_l(\Phi, Y, G_1) - \o{CP}_l(\Psi, Y, G_1)| \le ||\Phi - \Psi||$$
 $$|\u{CP}_u(\Phi, Y, G_1) - \u{CP}_u(\Psi, Y, G_1)| \le ||\Phi - \Psi||, \  |\o{CP}_u(\Phi, Y, G_1) - \o{CP}_u(\Psi, Y, G_1)| \le ||\Phi - \Psi||.$$
 \end{theorem}

 
 \textbf{Remark.} For  lower/upper condensed pressure we do not have an equivalent of Theorem \ref{invset}, since in general one cannot concatenate two balls $B_n(x, G_1, \vp)$ and $B_p(y, G_1, \vp)$ to create an $(n+p)$-condensed ball. 
 $\hfill\square$
 
 The following Theorem can be proved similarly as above. 
 
 \begin{theorem}\label{proptexh}
 In the above setting, the conclusions of Theorem \ref{propcond} and  Theorem \ref{osccond} are satisfied for the trajectory pressure $P(\Phi, Y, G_1, \om)$, and for its associated capacities, for any $\om \in \Sigma_m^+$. The same conclusions are satisfied also for  the lower and upper exhaustive pressures, and for their associated capacities. 
 \end{theorem}
 
 
 We investigate now the invariance properties of the amalgamated pressure over an arbitrary  set of trajectories (see (\ref{caly})). 

\begin{theorem}\label{invy}
In the above setting, consider a set $\cal Y \subset \Sigma_m^+ \times X$, and let $\Phi$ be a continuous multi-potential on the compact metric space $X$. Assume there exists $M>1$ such that $Card(f_j^{-1}(x)) \le M, 1 \le j \le m, x \in X$. Then,  $$P^A(\Phi, \cal Y, G_1) = P^A(\Phi, F(\cal Y), G_1)$$
 \end{theorem}
 
 \begin{proof}
 First let us denote $Y = \pi_2(\cal Y) \subset X$, and for any $\vp>0$ and $N >1$, take a countable family $\Gamma \subset \cal Y$ such that $Y \subset \mathop{\cup}\limits_{(\om, y) \in \Gamma} B_{n(\om, y)}(y, \om, \vp)$, and $n(\om, y) >N$ for all $(\om, y) \in \Gamma$.
 Recall that $F(\om, y) = (\sigma(\om), f_{\om_0}(y))$, so if we take $F(\Gamma)$, then we  obtain a set of trajectories $\Gamma'$ which can be used for the amalgamated pressure corresponding to $F(\cal Y)$. But the consecutive sums of $\Phi$, $S_n\Phi(y, \om)$ and $S_{n-1}\Phi(f_{\om_0}(y), \sigma\om)$  differ by only one term. Hence by taking limits when $N \to \infty$ and then $\vp \to 0$, one obtains $P^A(\Phi, \cal Y, G_1) \ge P^A(\Phi, F(\cal Y), G_1)$.

Vice-versa let a countable collection $\Gamma'$ so that the Bowen balls along the respective trajectories in $\Gamma'$ cover $\pi_2(F(\cal Y)$, and the lenghts of their trajectories are larger than $N$. Then using that each point $x \in X$ has at most $M$ $f_j$-preimages in $X$, for $1 \le j \le m$, it follows that we obtain a cover $\Gamma$ of $Y$ having at most $mM$ Bowen balls $B_{n+1}(x, \om, \vp)$ for each of the Bowen balls $B_n(y, \sigma\om, \vp)$ from $\Gamma$, with $f_{\om_0}(x) = y$. But in this case for every $(\om', y) \in \Gamma'$, and $\om\in \Sigma_m^+$ such that $\sigma \om = \om'$ and $x\in X$ such that $f_{\om_0}(x) = y$, we have that consecutive sums $S_{n+1}\Phi(x, \om)$ and $S_n\Phi(y, \sigma\om)$ differ by only one term. By taking  $N \to \infty$, and then  $\vp \to 0$, it follows $P^A(\Phi, \cal Y, G_1) \le P^A(\Phi, F(\cal Y), G_1)$. Thus from  above inequalities we obtain the formula.
\end{proof}
 
From the previous Theorem we obtain the following: 

 \begin{corollary}\label{trajominv}
Assume the maps $f_j \in G_1$ are surjective on $X$, and there exists $M >1$ so that each $x \in X$ has at most $M$ $f_j$-preimages in $X$, for $1 \le j \le m$. Then,  $$P(\Phi, G_1, \om) = P(\Phi, G_1, \sigma(\om)), \ \om \in \Sigma_m^+$$
\end{corollary}

 Now we  compare between the various types of pressures:
 
 \begin{theorem}\label{compac}
 In the above setting, if $Y\subset X$, then $$P^A(\Phi, Y, G_1) \le P_l(\Phi, Y, G_1) \le P_u(\Phi, Y, G_1)$$
 \end{theorem}
 
 \begin{proof}
 Let us take $N \ge 1$ and a cover of $Y$ with a family $\cal F$ of condensed balls $B_n(x, G_1, \vp)$, with lengths $n > N$. Then for each center $x$ of a condensed balls $B_n(x, G_1, \vp)$ from $\cal F$, choose the $n$-trajectory on which the consecutive sum of $\Phi$ is the smalest among all $n$-trajectories. This is  the trajectory $C$ on which  $s_n \Phi(x)$ is computed. For each center $x$ corresponding to a ball from $\cal F$, denote this minimizing trajectory by $C(x)$. 
 In this way we form a collection of balls along various trajectories (all of lengths larger than $N$), $\cal G = \{B_n(x, C(x), \vp), \ B_n(x, G_1, \vp) \in \mathcal F\},$
 which covers $Y$. From definitions, we  obtain
 $P^A(\Phi, Y, G_1) \le P_l(\Phi, Y, G_1).$
 \end{proof}

 \begin{theorem}\label{compeat}
 In the above setting, if $Y\subset X$ and $\om \in \Sigma_m^+$, then
 $$P^+_l(\Phi, Y, G_1) \le P^A(\Phi, Y, G_1) \le P(\Phi, Y, G_1, \om)$$
 In particular, for any $1 \le j \le m$, \ 
$P^+_l(\Phi, Y, G_1) \le P^A(\Phi, Y, G_1) \le P(\phi_j, Y, f_j).$
 \end{theorem}
 
 \begin{proof}
 From definitions, $P^A(\Phi, Y, G_1) \le P(\Phi, Y, G_1, \om)$. 
 Now let  an integer $N \ge 1$ and a cover $\cal F$ of $Y$ with sets $B_{n_i}(x, C, \vp)$, where $n_i$ is the length of $C\in\Sigma_m^*$, and $n_i \ge N$.
 Then since the exhaustive ball $B_{n_i}^+(x, G_1, \vp)$ contains $B_{n_i}(x, C, G_1, \vp)$, for any trajectory $C$, it follows that $Y$ is covered also with the sets $B_{n_i}^+(x, G_1, \vp)$ corresponding to the balls from $\cal F$. 
On the other hand for any $x \in X$ and $C\in \Sigma_m^*$, $|C| = n_i$, one has $s_{n_i}\Phi(x) \le S_{n_i}\Phi(x, C),$
 hence we obtain also  $P^+_l(\Phi, Y, G_1) \le P^A(\Phi, Y, G_1)$.
Also the usual pressure  $P(\phi_j, Y, f_j)$ of  $f_j$, is equal to trajectory pressure $P(\Phi, Y, G_1, (j, j, \ldots))$. Thus $P^+_l(\Phi, Y, G_1) \le P^A(\Phi, Y, G_1) \le P(\phi_j, Y, f_j),  1\le j \le m.$

 \end{proof}
 
 We now turn to \textbf{semigroups of toral endomorphisms}.

\ \  \textbf{Examples.}
 
 Let the compact metric space $X$ to be the 2-torus $\bb T^2$, and $f_1 = f_A: \bb T^2 \to \bb T^2$, $f_2 = f_B: \bb T^2 \to \bb T^2$ be  toral endomorphisms determined respectively by the matrices  
  $$A=\left(\begin{array}{ll}
                    0 & 1 \\

                    1 & a

                  \end{array} \right), \ \text{and} \ B=\left(\begin{array}{ll}
                    a & 1 \\

                    1 & 0

                  \end{array} \right),$$ where $a \in \bb Z, a \ge 2$.
 Then $A, B$ have the same eigenvalues $\lambda_1, \lambda_2$,  $|\lambda_1|< 1, |\lambda_2| >1$. 
 
 \begin{proposition}\label{exp}
 In the above setting, let $G_1 = \{id_{\bb T^2}, f_1, f_2\}$. Then, $$0 = h^+(G_1) = h^A(G_1) \ <  \ h(f_1) = h(f_2) \ < \  h(G_1),$$
 where $h^+(G_1)$ is the exhaustive entropy, $h^A(G_1)$ is the amalgamated entropy, and $h(G_1)$ is the condensed entropy for the semigroup $G$ generated by $G_1$ on $\bb T^2$.
 \end{proposition}
 
 \begin{proof}
From classical theory (for eg \cite{Wa}), it follows that $h(f_1) = \log|\lambda_2| = h(f_2)$, where $\lambda \in \{\lambda_1, \lambda_2\}$, and $0 < |\lambda_1| < 1, |\lambda_2| >1$. 
         Now, if we take the condensed ball $B_n(x, G_1, \vp)$ we notice that it is contained in a parallelogram with sides comparable to $|\lambda_2|^{-n}$, for any $n$.  Indeed the matrix $A$ has eigenvectors $\left(\begin{array}{ll}
                    1 \\

                    \lambda

                  \end{array} \right)$, where $\lambda \in \{\lambda_1, \lambda_2\}$, and $B$ has eigenvectors $\left(\begin{array}{ll}
                    \lambda \\

                    1

                  \end{array} \right)$, and we use them for Bowen balls $B_n(x, \om, \vp)$. To cover $\bb T^2$ we need at least $|\lambda_2|^{2n}$ such condensed balls, so the condensed pressure satisfies $$h(f_1) = h(f_2) = \log |\lambda_2| <  2 \log|\lambda_2| \le h(G_1)$$
 
 Next, we notice that $AB =\left(\begin{array}{ll}
                    1 & 0 \\

                    2a & 1

                  \end{array} \right),$
 which has a double eigenvalue equal to 1.
 And for any integer $n \ge 1$, $(AB)^n=\left(\begin{array}{ll}
                    1 & 0 \\

                    na & 1

                  \end{array} \right)$. 
                  Hence a Bowen ball $B_n(x, f_1f_2, \vp)$ for $f_1\circ f_2$, is a rectangle with a side of size $\vp$ and the other side of size $\frac{\vp}{na}$. Thus in order to cover $\bb T^2$ with Bowen balls along the trajectory $\om = (121212\ldots)$, we need $na/\vp$ such disjointed Bowen balls, for any $n >1$. This implies  $h(G_1, \om) = \mathop{\lim}\limits_{n \to \infty} \frac{\log (na/\vp)}{n} = 0$.
                So from Theorem \ref{compeat}, $0 \le h^+(G_1) \le h^A(G_1) \le h(G_1, \om) = 0$.              
                                              
 \end{proof}

 Next, we compute the exhaustive, amalgamated and condensed entropies. 
 
 \begin{proposition}\label{exp1}
 Let the toral endomorphisms $f_{A(\alpha, \beta)}: \bb T^2 \to \bb T^2$ determined  by the matrices $A(\alpha, \beta) =\left(\begin{array}{ll}
                    \alpha & 0 \\

                    0 & \beta

                  \end{array} \right)$, with $\alpha, \beta \in \bb N^*, \alpha, \beta \ge 2$. Let the semigroup $G$  generated by the finite set $G_1(\alpha, \beta, \gamma, \delta) = \{id, f_{A(\alpha, \beta)}, f_{A(\gamma, \delta)}\}$. Then, $$h^+(G_1(\alpha, \beta, \gamma, \delta)) = \log \min\{\alpha, \gamma\} + \log \min\{\beta, \delta\},$$   $$h^A(G_1(\alpha, \beta, \gamma, \delta)) = \min\{\log \alpha\beta, \log \gamma\delta\}, $$ $$h(G_1(\alpha, \beta, \gamma, \delta)) = \log \max\{\alpha, \gamma\} + \log \max\{\beta, \delta\}.$$
                  In particular if $\alpha = 2, \beta = 3, \gamma = 3, \delta = 2$, and $G_1:= G_1(2, 3, 3, 2), $ then, 
                  $$h^+(G_1) = \log 4 < h(f_{A(2, 3)}) = h^A(G_1) = \log 6 < h(G_1) = \log 9.$$
                 
                  \end{proposition} 
 
 \begin{proof}
 Fix $\alpha, \beta, \gamma, \delta$ and denote $G_1(\alpha, \beta, \gamma, \delta)$ by $G_1$ to simplify notation. 
  Denote also $$\ell_1 := \min\{\alpha, \gamma\}, \ell_2 := \min\{\beta, \delta\}, \ L_1:= \max\{\alpha, \gamma\},  L_2:= \max\{\beta, \delta\}$$
  In this case the exhaustive balls $B_n^+(x, G_1, \vp)$ are unions of $n$  rectangles centered at $x$ of sides:  $\alpha^{-n}\vp$ and $\beta^{-n}\vp$, then $\alpha^{-n+1}\gamma^{-1}\vp$ and $\beta^{-n+1}\delta^{-1}\vp, \ldots,$ and finally $ \gamma^{-n}\vp$ and $\delta^{-n}\vp$. 
  The factors of contraction/expansion of these sides are constant at each step.  Thus we see that $B_n^+(x, G_1, \vp)$ contains a rotated rectangle  of sides $C \ell_1^{-n}\vp$ and $C\ell_2^{-n}\vp$, and is contained in a square of sides $\tilde C\ell_1^{-n}\vp$ and $\tilde C \ell_2^{-n}$, for $n >1$, where $C, \tilde C$ are independent of $n$. Therefore, the number $N^+(n, \vp)$ of exhaustive balls $B_n^+(x, G_1, \vp)$ needed to cover $\bb T^2$, lies between $C_1 \ell_1^{n}\ell_2^n$ and $C_2 \ell_1^{n}\ell_2^n$, for any $n > 1$, where the constants $C_1, C_2$ are independent of $n$. Hence from definition, $h^+(G_1) = \log \ell_1\ell_2$. 
 
For the amalgamated entropy, denote $f_1 = f_{A(\alpha, \beta)}, f_2 = f_{A(\gamma, \delta)}$. Notice that if $\om|_n = (i_1, \ldots, i_n)$ is a trajectory, then the Bowen ball $B_n(x, \om, \vp)$ is a rectangle of sides $\alpha^{-p}\gamma^{-n+p}\vp$ and $\beta^{-p} \delta^{-n+p}\vp$, where $p$ is the number of indices $j, 1\le j \le n$ with $i_j = 1$. But the area of such a rectangle is comparable with $(\alpha\beta)^{-p} (\gamma\delta)^{-n+p}\vp^2$, thus we need $C (\alpha\beta)^p(\gamma\delta)^{n-p}$ such rectangles to cover the torus $\bb T^2$, where the constant $C$ is independent of $n$.  Let $\cal F$ be a family of Bowen balls along various trajectories which covers $\bb T^2$. Denote by $N_k(\cal F)$ the number of Bowen balls $B_n(y, \om, \vp)\in \cal F$ where $\om|_n$ contains exactly $k$ indices equal to 1. Hence, $N_0(\cal F) (\gamma\delta)^{-n} + N_1(\cal F)(\alpha\beta)(\gamma\delta)^{-n+1} + \ldots N_n(\cal F)(\alpha\beta)^{-n} \ge C>0$, with $C$ independent of $n$. Now in case $\alpha\beta < \gamma\delta$, then $$Card(\cal F) = N_0(\cal F) + \ldots + N_n(\cal F) \ge C (\alpha\beta)^n$$ And the value $C(\alpha\beta)^n$ is obtained if all the Bowen balls in $\cal F$ are along the trajectory $(1, \ldots, 1)$. Hence, $\mathop{\inf}\{ Card(\cal F), \ \cal F \ (n, \vp)-\text{covers} \ \bb T^2\} = C(\alpha\beta)^n$, for any $n >1$. Similarly, if $\gamma\delta \le \alpha\beta$, then the minimal number of Bowen balls needed to cover $\bb T^2$ is comparable to $(\gamma\delta)^n$. These facts imply that, $h^A(G_1) = \min\{\log\alpha\beta, \log \gamma\delta\}$. 

For the condensed entropy, we take $B_n(x, G_1, \vp) = \mathop{\cap}\limits_{|\om| = n} B_n(x, \om, \vp)$. From above, recall that a Bowen ball $B_n(x, \om, \vp)$ is a rectangle of sides $\alpha^{-p}\gamma^{-n+p}\vp$ and $\beta^{-p} \delta^{-n+p}\vp$, where $p$ is the number of $j, 1\le j \le n$ with $i_j = 1$. Thus a condensed ball $B_n(x, G_1, \vp)$ is a rectangle with sides $L_1^{-n}\vp$ and $L_2^{-n}\vp$. We need $C (L_1L_2)^n$ such condensed balls to cover $\bb T^2$ for $n \ge 1$, with $C$ independent of $n$. Thus $h(G_1) = \log L_1L_2$.

 \end{proof}
 
 In particular, the condensed, amalgamated, and exhaustive entropies  can be used to distinguish between \textbf{classes} of  semigroup actions.
 
 \begin{definition}\label{topconj}
 If $G$ is a set of self-maps on $X$ and $g:X \to Z$ is a homeomorphism of compact metric spaces, then $G$ and $G' := \{g\circ f\circ g^{-1}, f \in G\}$ are called \textbf{topologically conjugate}.
 \end{definition}
 

\textbf{Example.}

Let  two sets of toral endomorphisms $G_1 = \{f_{A(4, 5)}, f_{A(2, 6)}\}$, and $G_2  = \{f_{A(2, 10)}, f_{A(3, 4)}\}$ on $\bb T^2$. The (usual) topological entropies of the respective maps are equal, $h(f_{A(4, 5)}) = h(f_{A(2, 10)})$,  $h(f_{A(2, 6)}) = h(f_{A(3, 4)})$. However from Theorem \ref{proptexh}, the sets $G_1, G_2$ are not conjugated, and thus also the semigroups generated by them are not conjugated, since from Proposition \ref{exp1} the exhaustive entropies are different, $$h^+(G_1) = \log 10 \ne h^+(G_2) = \log 8. $$ $\hfill\square$

 Our results are related also to the problem of \textbf{invariant closed sets} in $\bb T^k, k \ge 2$ for \textbf{abelian} semigroups of toral endomorphisms studied in \cite{Be}. 
 
\begin{corollary}\label{bere}
 Let a commutative semigroup $G$ of toral endomorphisms on $\bb T^2$ with a finite generator set $G_1$. Assume that all the eigenvalues of the maps in $G_1$ are larger than 1 in absolute value, and that  there exists  $f\in G_1$ with its characteristic polynomial irreducible over $\bb Z$ and its eigenvalues $\lambda_1, \lambda_2$  different in absolute value. Assume also  $$h^+(G_1) < h(f), \  \forall f \in G_1.$$ Then, the only $G$-invariant closed infinite set in $\bb T^2$ is $\bb T^2$ itself. 
 \end{corollary}
 
 \begin{proof}
 If the characteristic polynomial of $f$ is irreducible over $\bb Z$ but the characteristic polynomial of $f^n$ is reducible for some $n \ge 2$, then $f^n$ must have a double eigenvalue, i.e $\lambda_1^n = \lambda_2^n$.  But we assumed that the eigenvalues $\lambda_1, \lambda_2$ of $f$ have different absolute values. So the characteristic polynomial  of $f^n$ is irreducible over $\bb Z$ for any $n \ge 1$. Also we assumed that the eigenvalues of $f$ are larger than 1 in absolute values.  Hence the first two conditions in Theorem 2.1 in \cite{Be} are satisfied.
 
 Now, we will show  that there exists a pair of rationally independent endomorphisms in $G$.  First, in the above setting it follows from  Lemma 3.4 in \cite{Be} that there exists a basis $\{v, w\}$ in $\bb C^2$ so that every endomorphism from $G$ is diagonalizable with respect to this basis.  By using Lebesgue measures of parallelograms, as in the proof of Proposition \ref{exp1} above and \cite{Wa}, we see that $h^+(G_1) = \log |\lambda_{1 *}| + \log |\lambda_{2 *}|$, where $|\lambda_{1 *}|$ is the smallest eigenvalue (in absolute value) of a map in $G_1$ in the $v$-direction, and $|\lambda_{2 *}|$ is the smallest eigenvalue (in absolute value) of a map in $G_1$ in the $w$-direction.  
 But if $h^+(G_1) < h(f), f \in G$, then $\lambda_{1*}$ and $\lambda_{2*}$ cannot be eigenvalues of the same map. Hence $\lambda_{1*}$ is an eigenvalue for some toral endomorphism $g\in G_1$, i.e $\lambda_{1*} = \lambda_1(g)$. And $\lambda_{2*}$ is an eigenvalue for another toral endomorphism $\kappa\in G$, i.e $\lambda_{2*} = \lambda_2(\kappa)$.
 Assume there are integers $n, p \ge 1$ such that $g^n = \kappa^p$. Then, $$|\lambda_{1*}| = |\lambda_1(g)| = |\lambda_1(\kappa)|^{p/n}, |\lambda_2(g)| = |\lambda_2(\kappa)|^{p/n} = |\lambda_{2*}|^{p/n}$$
 We assumed $h^+(G_1) < h(f), f \in G$. Thus $\exp(h^+(G_1)) = |\lambda_{1*}\lambda_{2*}| < \exp(h(g)) = |\lambda_1(g)\lambda_2(g)| = |\lambda_{1*}| |\lambda_{2*}|^{p/n}$. But $|\lambda_{2*}| >1$, so $p >n$. However, $\exp(h^+(G_1)) = |\lambda_{1*}\lambda_{2*}| < \exp(h(\kappa)) = |\lambda_1(\kappa)\lambda_2(\kappa)| =  |\lambda_{1*}|^{n/p} |\lambda_{2*}|$, hence $p < n$.
This is a contradiction, so   $\exists \kappa \in G$ such that $g$ and  $\kappa$ are rationally independent. 
Thus all three conditions from Theorem 2.1 of \cite{Be} are satisfied, hence the only $G$-invariant closed infinite set   is $\bb T^2$ itself.

 \end{proof}

Next, let us recall  the notion of topological pressure of single potentials for free semigroup actions, or \textbf{free topological pressure}, studied for eg in \cite{LMW} and \cite{RV},   which generalizes the notion of free entropy introduced in \cite{Bu}. One considers averages of  minimizing sums over all $m^n$ trajectories of length $n$, namely if $\phi \in \mathcal C(X, \bb R)$ then 
$$P_{free}(\phi, G_1) = \mathop{\lim}\limits_{\vp \to 0} \mathop{\limsup}\limits_{n\to \infty} \frac{1}{m^n} \mathop{\sum}\limits_{|C| = n}\inf\{\mathop{\sum}\limits_{x \in \cal F_C} e^{S_n\phi(x, C)}, X = \mathop{\bigcup}\limits_{x \in \cal F_C} B_n(x, C, \vp)\}
$$
 It is proved  easily that the  free topological pressure is computed also with separated sets.
 Similarly one  defines free topological pressure for a multi-potential $\Phi$,
 \begin{equation}
 P_{free}(\Phi, G_1) = \mathop{\lim}\limits_{\vp \to 0} \mathop{\limsup}\limits_{n\to \infty} \frac{1}{m^n} \mathop{\sum}\limits_{|C| = n}\inf\{\mathop{\sum}\limits_{x \in \cal F_C} e^{S_n\Phi(x, C)}, X = \mathop{\bigcup}\limits_{x \in \cal F_C} B_n(x, C, \vp)\}
\end{equation}
 
 The amalgamated pressure, free pressure, and upper condensed pressure satisfy:
 
 \begin{theorem}\label{freetp}
 In the above setting, we have
 $$P^A(\Phi, G_1) \le P_{free}(\Phi, G_1) \le P_u(\Phi, G_1)$$
 \end{theorem}
 
 \begin{proof}
 Using Theorem \ref{invset} we know that the amalgamated pressure can be computed as a capacity, using $(n, \vp)$-spanning sets at each level $n$. But by definition, at each level $n$, the respective minimizing sum in the definition of the amalgamated pressure is smaller than or equal to each of the $m^n$ minimizing sums along the possible trajectories $C$ of length $n$. This implies that $P^A(\Phi, G_1) \le P_{free}(\Phi, G_1)$.
 
 For the second inequality, notice that if $\cal F$ is a cover of $X$ with condensed balls, then for any trajectory $\om \in \Sigma_m^+$, we obtain a cover of $X$ with the Bowen balls $B_n(x, \om, \vp), x \in \cal F$.
 Now in the definition of the upper condensed pressure $P_u$ we take for each point $x \in \cal F$ the trajectory $\om$ which maximizes the consecutive sum $S_n\Phi(x, \eta)$, for all $\eta \in \Sigma_m^+$. Hence for any $\eta \in \Sigma_m^+$ we have, in the notation of (\ref{lcs}),
 $$S_n\Phi(x, \eta) \le S_n\Phi(x)$$
 Thus each of the $m^n$ sums at level $n$ in the definition of $P_{free}(\Phi, G_1)$  is smaller than  the sum at level $n$ in  definition of $P_u(\Phi, G_1)$. So $P_{free}(\Phi, G_1) \le P_u(\Phi, G_1).$
 \end{proof}
 
 Now, given the finite set of generators $G_1$ for a semigroup of maps on $X$, and a subset $Y \subset X$, we introduce the \textbf{trajectory pressure map} on $Y$, 
\begin{equation}\label{TP} 
TP_Y: \Sigma_m^+ \to \bb R, \ TP_Y(\om) = P(\Phi, Y, G_1, \om)
\end{equation} 
$TP_Y$ is a measurable map, but in general it is \textbf{discontinuous} on $\Sigma_m^+$. 
 

 However,  the trajectory pressure map is continuous on some sets of $\Sigma_m^+$:
 
 \begin{corollary}\label{conste}
 If $\nu$ is a $\sigma$-invariant ergodic probability measure on $\Sigma_m^+$ and if $Y \subset X$ and  $f_j(Y) = Y, 1 \le j \le m$, then the map $TP_Y(\cdot)$ is constant $\nu$-a.e. 
 \end{corollary}
 
\begin{proof}
For any $\om \in \Sigma_m^+$,  take $\cal Y = \{\om\}\times Y$. Then from Theorem \ref{invy}, and using that $f_j(Y) =Y, 1 \le j \le m$, it follows that $TP_Y(\om) = TP_Y(\sigma\om), \forall \om \in \Sigma_m^+$. Then from  ergodicity of $\nu$, $TP_Y(\cdot)$ is constant $\nu$-a.e. on $\Sigma_m^+$.

\end{proof}

 \section{Inverse limit of the semigroup dynamics. }
 
 Let  as before a compact metric space $X$ and a semigroup of maps on $X$ generated by a finite set $G_1 = \{id_X, f_1, \ldots, f_m\}$, with $f_i: X \to X$ are continuous maps for $1 \le i \le m$.
 Then, inspired by the notion of inverse limit of an endomorphism (see for eg \cite{Ru1}, \cite{M-MZ}), we define the \textbf{inverse limit associated to the generators $G_1$},
 \begin{equation}\label{mixil}
 \hat X_{G_1} := \big\{(\eta, \hat x(\eta)) \in \Sigma_m \times X^{\bb Z},  \text{with} \ \hat x(\eta) = (\ldots, x_{-1}, x, x_1, \ldots), \ x_j = f_{\eta_{j-1}}(x_{j-1}), j \in \bb Z \big\}
 \end{equation}
 
$\hat X_{G_1}$ is a compact metric space with the metric  from $\Sigma_m^+ \times X^{\bb Z}$. 
 The shift map $$\hat\sigma: \hat X_{G_1} \to \hat X_{G_1}, \ \hat\sigma(\om, \hat x(\om)) = (\sigma\om, \sigma\hat x(\om)),$$
 is a homeomorphism. 
Denote the space $Z = \Sigma_m^+ \times X$ and consider the endomorphism $$F: Z \to Z, \ F(\om, x) = (\sigma(\om),  f_{\om_0}(x))$$
This endomorphism has proved to be very useful, for eg in \cite{Bu}, \cite{CRV}, \cite{LMW}.
We show that $\hat X_{G_1}$ is  in fact the inverse limit (in the usual sense) of $F$ on $Z$. 

\begin{proposition}\label{invl}
The inverse limit $\hat F: \hat Z \to \hat Z$ of the endomorphism $F:Z \to Z$, is topologically conjugate  to the shift $\hat \sigma: \hat X_{G_1} \to \hat X_{G_1}$.
\end{proposition}

\begin{proof}
The inverse limit of $F$ on $Z$ is the space $\hat Z = \{\hat z = (z, z_{-1}, z_{-2}, \ldots), F(z_{-i}) = z_{-i+1}, i \ge 1\}$. So if we start with $z = (\om, x) \in Z = \Sigma_m^+ \times X$, then $z_{-1} = (\om_{-1}, x_{-1}) \in Z$ with $F(z_{-1}) = z$. 
Assume that $\om_{-1} = (\om_{-1, 0}, \om_{-1, 1}, \ldots)$ and $\sigma(\om_{-1}) = \om = (\om_0, \om_1, \ldots)$. Then $\om_{-1} = (\om_{-1, 0}, \om_0, \om_1, \ldots)$. If $\om_{-2} = (\om_{-2, 0}, \om_{-2, 1}, \ldots)$, then since $\sigma(\om_{-2}) = \om_{-1}$, it follows that $\om_{-2, 1} = \om_{-1, 0}, \om_{-2, 2} = \om_0, \ldots$. We continue this procedure by which a coordinate is added in front of the previous trajectory. Define  $\eta \in \Sigma_m$, $\eta = (\ldots, \eta_{-1}, \eta_0, \eta_1, \eta_2, \ldots)$ with $\eta_{-1} = \om_{-1, 0}, \eta_0 = \om_0, \eta_{-2} = \om_{-2, 0}, \ldots$.
Since $F(z_{-1}) = z$, then $f_{\om_{-1, 0}}(x_{-1}) = x$. As $F(z_{-2}) = z_{-1}$, we have $f_{\om_{-2, 0}}(x_{-2}) = x_{-1}$, and so on. 
Thus $(\eta, (\ldots, x_{-2}, x_{-1}, x_0, x_1, x_2, \ldots)) \in \Sigma_m \times X^{\bb Z}$ belongs to $\hat X_{G_1}$. This defines a topological conjugacy $\Theta$ between $\hat F : \hat Z \to \hat Z$ and  $\hat \sigma: \hat X_{G_1} \to \hat X_{G_1}$. 
\end{proof}

Given now  $\Phi$ as above, consider the potentials, $$\hat \Phi: \hat X_{G_1} \to \bb R, \ \hat \Phi(\eta, \hat x(\eta)) = \phi_{\eta_0}(x_0), \  \text{and},$$
$$\Phi^+: \Sigma_m^+ \times X \to \bb R, \ \Phi^+(\om, x) := \phi_{\om_0}(x)$$

 \begin{proposition}\label{lift}
 In the above setting, for any multi-potential $\Phi$ on $X$,  $$P(\Phi^+, F) = P(\hat \Phi, \hat \sigma)$$
 \end{proposition}
 
 \begin{proof}
Denote as above by $Z = \Sigma_m^+\times X$ and $\hat Z$ be the inverse limit relative to the endomorphism $F$, and let $\pi_Z:\hat Z \to Z$ the canonical projection.  Then by \cite{Ru2} the pressure of $\hat\Phi$ with respect to $F$ is equal to the pressure of $\Phi^+ \circ \pi_Z$ with respect to  $\hat F$,  $P(\Phi^+, F) = P(\Phi^+ \circ \pi_Z, \hat F)$.
From Proposition \ref{invl} and its proof, $\Phi^+\circ \pi_Z \circ \Theta^{-1} = \hat\Phi$, and $\Theta$ conjugates  $\hat F$ on $\hat Z$ with  $\hat\sigma$ on $\hat X_{G_1} $. So $P(\Phi^+, F) = P(\Phi^+\circ \pi_Z, \hat F) = P(\hat\Phi, \hat \sigma)$.
\end{proof}
 
 We now compare the pressure $P(\hat\Phi, \hat\sigma)$ on $\hat X_{G_1}$ (which from  Theorem \ref{lift} is equal to $P(\Phi^+, F)$), with the amalgamated pressure  and  condensed pressure of $\Phi$ on $X$.
  
 \begin{theorem}\label{Fam}
 In the above setting, 
  $$ P^a(\Phi, G_1) + \log m \le P(\Phi^+, F) = P(\hat\Phi, \hat\sigma) \le P_u(\Phi, G_1) + \log m$$
 \end{theorem}
 
\begin{proof}
Since $F: \Sigma_m^+ \times X \to \Sigma_m^* \times X, \ F(\om, x) = (\sigma \om, \phi_{\om_0}(x)) $, one sees that a Bowen ball for $F$ has the form $[\om_1\ldots \om_n] \times B_n(x, \om, \vp)$. So an $(n, \vp)$-spanning set $\cal F$ for $F$ contains balls of type $[\om_1\ldots \om_n]\times B_n(x, \om, \vp)$, where for every $\om \in \Sigma_m^+$ the respective balls $B_n(x, \om, \vp)$ form a set $\cal F_n(\om)$ which is $(n, \vp)$-spanning on $X$ along the trajectory $\om$. 
Also notice that for any $(\om, x) \in \Sigma_m^+ \times X$, $S_n\Phi^+(\om, x) = S_n\Phi(x, \om)$.

But from definitions it is then clear that $$\mathop{\sum}\limits_{x \in \mathcal F_n(\om)} \exp(S_n\Phi^+(\om, x)) \ge \inf \{\mathop{\sum}\limits_{(C, y) \in \cal G} \exp(S_n\Phi(y, C)), \ X  = \mathop{\cup}\limits_{(C, y) \in \cal G} B_n(y, C, G_1, \vp) \}
$$
 Thus, since for any $n \ge 1$ there are exactly $m^n$ cylinders of type $[\om_1\ldots\om_n]$ which cover $\Sigma_m^+$. By using Theorem \ref{invset}, we obtain that $P(\Phi^+, F) \ge \log m + P^A(\Phi, G_1).$

 To prove the other inequality in the statement,  let us now consider an integer $N\ge 1$ and a cover $\Gamma$ of $X$ with condensed balls $B_{n_i}(x, G_1, \vp)$, with $n_i \ge N$. And for any $\om\in\Sigma_m^+$ and any ball in $\Gamma$ take the product $[\om_1\ldots \om_{n_i}] \times B_{n_i}(x, \om, \vp)$, which is a Bowen ball for $F$;  denote the collection of all these Bowen balls by $\cal F(\Gamma)$.

From above it follows that $\cal F(\Gamma)$ covers $\Sigma_m^+\times X$, and for each $B_{n_i}(x, G_1, \vp) \in \Gamma$ there are $m^{n_i}$ balls of type $[\om_1\ldots \om_{n_i}]\times B_{n_i}(x, \om, \vp)$. 
On the other hand, we have that $$\mathop{\sum}\limits_{\cal F(\Gamma)} \exp(S_{n_i}\Phi^+(\om, x)) \le \mathop{\sum}\limits_{\Gamma} \exp(S_{n_i}\Phi(x, G_1) + n_i \log m)$$
Thus, $P(\Phi^+, F) \le P_u(\Phi, G_1) + \log m.$

\end{proof}

 \section{Partial Variational Principle for amalgamated pressure.}

 In this Section we prove a Partial Variational Principle for the amalgamated pressure of multi-potentials in the dynamics of semigroups for certain hyperbolic maps. Variational principles for pressure and for other dynamical notions have been obtained by many authors (for eg \cite{Wa},  \cite{Pe}, \cite{Ru1}, \cite{T}).
 
Recall the notations of previous Sections, and the skew-product $F$ from (\ref{FF}).  Consider a probability $F$-invariant measure $\mu$ on $\Sigma_m^+ \times X$. Its canonical projection on the second coordinate is called a \textbf{marginal measure} on $X$. 
On $X$ one can define also \textit{stationary measures}, i.e measures $\nu$ for which there exists some shift-invariant probability $\rho$ on $\Sigma_m^+$ so that $\rho\times \nu$ is $F$-invariant. Stationary measures are marginal, however there may exist marginal measures which are not stationary. In this Section, we will employ only $F$-invariant measures $\mu$ on the lift $\Sigma_m^+\times X$, as they encapsulate in general more information than stationary measures (for eg \cite{Ar}, \cite{Ki}).

Denote by $G_\mu$ the set of \textbf{generic points for $\mu$}, i.e the set of points for which the Birkhoff Ergodic Theorem holds with respect to $\mu$ for every real-valued continuous function on $X$. The set $G_\mu$ is Borel in $\Sigma_m^+ \times X$, but may  be non-compact. 

We now define the measure-theoretic amalgamated entropy of $\mu$, recalling the definition of amalgamated pressure over a set $\cal Y \subset \Sigma_m^+\times X$ from (\ref{caly}); in our case $\cal Y = G_\mu$.

\begin{definition}\label{ament}
In the above setting, define the \textbf{amalgamated measure-theoretic entropy} $h^A(\mu, G_1)$ to be the amalgamated topological entropy on the projection of $G_\mu$, $$h^A(\mu, G_1) := h^A(\pi_2(G_\mu), G_\mu, G_1) = P^A(0, G_\mu, G_1).$$
\end{definition}

 Recall from Remark \ref{calyne} that, in general we cannot say which of the quantities $P^A(\Phi, G_1)$ and $P^A(\Phi, G_\mu, G_1)$ is larger.

 \begin{theorem}\textbf{Partial Variational Principle for $P^A$.} \label{PVP}
 
 a) Assume $X$ is a compact set in  $\bb R^D$ and all the maps $f_j \in G_1$ are $\mathcal C^1$-differentiable and conformal on a neighbourhood of $X$, and there exists a constant $\alpha >1$ so that $|D f_j(x)| > \alpha >1, \forall x \in X$, $1\le j \le m$. Then for any continuous multi-potential $\Phi$, $$P^A(\Phi, G_1) \le \sup\{h^A(\mu, G_1) + \int \Phi^+  \  d\mu, \ \mu \ F-\text{invariant probability on} \ \Sigma_m^+\times X\}$$

 b) The same conclusion as in a) holds,  if all functions $f_j$ are hyperbolic over a $G_1$-invariant set $\Lambda\subset \bb R^D$ and if for every $x \in \Lambda$ the stable tangent spaces $E^s(f_j, x)$ coincide, i.e $E^s(f_j, x) = E^s(f_k, x), 1 \le j, k \le m$, and the maps $f_j, 1 \le j \le m$ are conformal on the common local stable manifolds $W^s(x), x \in \Lambda$.  
 \end{theorem} 
 
 \begin{proof}
 We apply first the classical Variational Principle for the map $F$ and potential $\Phi^+$ on $\Sigma_m^+ \times X$; recall  $\Phi^+(\om, x) = \phi_{\om_0}(x), (\om, x) \in \Sigma_m^+ \times X$. 
 Therefore we obtain, 
 \begin{equation}\label{cvp}
 P(\Phi^+, F) = \sup\{h(\mu) + \int\Phi^+ \ d\mu, \ \mu \ F-\text{invariant probability}\}
 \end{equation}

 Let now $\mu$ be an $F$-invariant probability measure on $\Sigma_m^+\times X$. Since $G_\mu$ is the set of generic points with respect to $\mu$ for every real-valued continuous function on $X$, we have that for every $(\om, x) \in G_\mu$ and every $\vp >0$, there exists an integer  $n(\vp)$ such that $|\frac{S_n\Phi^+(\om, x)}{n} - \int\Phi^+ \ d\mu| < \vp, \forall n > n(\vp)$.

 Denote by $\psi:\Sigma_m^+ \times X \to \bb R, \psi(\om, x) = \log|Df_{\om_0}(x)|$.
 Then since all the points from $G_\mu$ are generic for $\mu$, we have that $G_\mu = \mathop{\cap}\limits_{\vp >0} \mathop{\cup}\limits_{N >1} G_\mu(N, \vp)$, where $$G_\mu(N, \vp) := \{(\om, x) \in G_\mu, \ |\frac{S_n\Phi^+ (\om, x) }{n} - \int \Phi^+ \ d\mu| < \vp, |\frac{S_n\psi(\om, x)}{n} - \int\psi d\mu| < \vp,  n \ge N\}$$ We will work first on $G_\mu(N, \vp)$. Recall that the functions $f_j, 1 \le j \le m$ are conformal and expanding on a neighbourhood of $X$. So if $n > N$, by using the definition of $G_\mu(N, \vp)$ it follows that  the various balls $B_n(x, \om, \vp)$ for $(\om, x) \in G_\mu(N, \vp)$ have the same radius modulo a factor of $e^{n\vp}$.  
 
 Now, according to Theorem \ref{invset} the amalgamated pressure $P^A$ can be computed using covers with Bowen balls $B_n(x, \om, \vp)$ with the same length of trajectory $n$. 
Let a Bowen ball $B_n(x, \om, \vp)$ where $(\om, x) \in G_\mu(N, \vp)$ and assume $n>N$. Then  due to the fact that all functions $f_j, 1 \le j \le m$ are $\mathcal C^1$-differentiable on $X$ and expanding, there exists a maximum integer $k_n$ such that  $B_n(x, \om, \vp) \subset B_{k_n}(x, \eta, \vp)$, and where $$(1-\vp)n \le k_n \le (1+\vp)n,$$ for any  $(\eta, x) \in G_\mu(N, \vp)$.  Notice that a Bowen ball for $F$ has the form $[\om_1\ldots \om_n] \times B_n(x, \om, \vp)$. 
 If $\cal F$ is a countable cover of $\pi_2(G_\mu(N, \vp))$ with Bowen balls of type $B_n(x, \om, \vp)$ used to compute $P^A(\Phi, G_\mu(N, \vp), G_1)$, then we obtain a cover $\cal G$ of $G_\mu(N, \vp)$ with Bowen balls for the skew-product $F$ on $\Sigma_m^+ \times X$, $$\cal G = \{[\eta_1 \ldots \eta_{k_n}]\times B_{k_n}(x, \eta, \vp), \text{where} \ B_n(x, \om, \vp) \in \cal F, \ (\eta, x) \in G_\mu(N, \vp)\}$$

  Moreover, we have $S_n\Phi(x, \om) = S_n\Phi^+ (\om, x),$ where the second term above is the consecutive sum with respect to $F$. 
  We need  $m^{k_n}$ cylinders of type $[\eta_1 \ldots \eta_{k_n}]$ to cover $\Sigma_m^+$; thus since $k_n \le (1+\vp) n$, we cover $\Sigma_m^+$ with at most $m^{(1+\vp)n}$ such cylinders. On the other hand, $$|S_{k_n}\Phi(x, \eta) - S_n\Phi(x, \eta)| \le \vp n ||\Phi||,$$ and if $(\eta, x) \in G_\mu(N, \vp)$ and $n> N$, then $|\frac{S_n\Phi(x, \eta)}{n} - \int\Phi^+\ d\mu| < \vp$. But  for any cover $\cal F$ as above we obtain a cover $\cal G$ of $G_\mu(N, \vp)$. Therefore, by using the definition of amalgamated pressure over the set $G_\mu$ in (\ref{caly}) and the last inequalities, we infer that $$P(\Phi^+, G_\mu(N, \vp)) \le P^A(\Phi, G_\mu(N, \vp), G_1) + (1+\vp) \log m + \vp||\Phi||$$ Then, by taking $N\to \infty$ and $\vp \to 0$, it follows from above and Theorem \ref{amal} that 
\begin{equation}\label{pstar}
P(\Phi^+, G_\mu)  \le P^A(\Phi, G_\mu, G_1) + \log m
\end{equation}
 On the other hand, from \cite{Pe}, it follows that for $F$ and $\Phi^+$ on $\Sigma_m^+\times X$ we have, $$P(\Phi^+, G_\mu) = h(\mu) + \int_{\Sigma_m^+\times X} \Phi^+ \ d\mu,$$ where $P(\Phi^+, G_\mu)$ is the (usual) topological pressure of $\Phi^+$ with respect to $F$ on $G_\mu$. 
 
 Denote by $\cal M(F)$ the set of $F$-invariant probability measures on $\Sigma_m^+\times X$. 
 
We now use (\ref{pstar}) and  the Variational Principle for $F$ and $\Phi^+$ from (\ref{cvp}), hence $$P(\Phi^+, F) = \sup\{P(\Phi^+, G_\mu), \mu \in \cal M(F)\}\le \sup\{P^A(\Phi, G_\mu, G_1), \mu \in \cal M(F) \} + \log m$$
However from Theorem \ref{Fam}, $P^A(\Phi, G_1) + \log m \le P(\Phi^+, F)$. These inequalities imply, 
\begin{equation}\label{p*F}
P^A(\Phi, G_1) \le  \sup\{P^A(\Phi, G_\mu, G_1), \ \mu \in \cal M(F)\}
\end{equation}

But  the trajectories in $G_\mu$ are generic for $\Phi^+$ with respect to $\mu$, and  the same open covers of $\pi_2(G_\mu)$ can be used to compute both $h^A$ and $P^A$. Therefore, using the Birkhoff Ergodic Theorem  for $F$ and $\Phi^+$ on $\Sigma_m^+\times X$ along generic trajectories, we obtain that $P^A(\Phi, G_\mu, G_1) = h^A(\pi_2(G_\mu), G_\mu, G_1) + \int\Phi^+ d\mu$. Hence from  (\ref{p*F}),
$$P^A(\Phi, G_1) \le \sup\{h^A(\mu, G_1) + \int \Phi^+ \ d\mu, \ \mu \in \cal M(F)\}$$

The proof for b) is similar. Since for every $x \in \Lambda$ the stable direction is common and denoted by  $E^s(x)$, then the local stable manifolds of $f_j, 1 \le j \le m$ coincide at $x$ and are equal to some $W^s(x, \vp)$. Hence the Bowen balls $B_n(x, \om, \vp)$ are tubular neighbourhoods along $W^s(x, \vp)$ for any trajectory  $\om \in \Sigma_m^+$. Thus one can cover a set $A \subset \Lambda$ with a family $\cal F$ of Bowen balls $B_n(x, \eta, \vp)$ if and only if, for any submanifold $\Delta$ transverse to $E^s(z)$ at every $z \in \Lambda$, one can cover the set $A\cap \Delta $ with the intersections of balls $\Delta \cap B_n(x, \eta, \vp)$, where $B_n(x, \eta, \vp) \in \cal F$.
  
Thus for $N$ large and $\vp>0$, given a countable cover $\cal F$ with $B_n(x, \om, \vp)$, $n>N$ used to compute $P^A(\Phi, G_\mu(N, \vp), G_1)$, we obtain a cover $\cal G$ of $G_\mu(N, \vp)$ with sets $[\eta_1\ldots \eta_{k_n}] \times B_{k_n}(x, \eta, \vp)$  for all $B_n(x, \om, \vp) \in \cal F$ and $\eta$ with $(\eta, x) \in G_\mu(N, \vp)$. Hence $$(1-\vp) n \le k_n \le (1+\vp)n$$ So for every $B_n(x, \om, \vp)$ from $\cal F$, one has at most $m^{(1+\vp)n}$ Bowen balls  $[\eta_1\ldots \eta_{k_n}] \times B_{k_n}(x, \eta, \vp)$ for $F$ in the collection $\cal G$.
Then we apply the same argument as above, by using the Birkhoff Ergodic Theorem for $F$ and $\Phi^+$  on $\Sigma_m^+\times \Lambda$. Therefore one obtains, $P^A(\Phi, G_1) \le \sup\{h^A(\mu, G_1) + \int \Phi^+ \ d\mu, \ \mu \in \cal M(F)\}.$
\end{proof}

 \section{Local entropies of measures for semigroups.}
 
 In this Section, by analogy with the local entropies of measures introduced by Brin and Katok in the classical case (\cite{BK}, \cite{Wa}), we study several notions of local entropies for measures $\mu$ on $X$, by taking in consideration the dynamics of the semigroup $G$ generated by a finite set $G_1$.  The measures are not assumed to be $G$-invariant.

 \begin{definition}\label{locale}
 Consider $\mu$ to be a probability measure on $X$. For every point $x\in X$, define the \textbf{upper/lower local amalgamated entropy} at $x$, respectively by:
 $$h_\mu^u(x, G_1) = \mathop{\lim}\limits_{\vp \to 0}\mathop{\u{\lim}}\limits_{n \to \infty} \mathop{\sup}\limits_{|\om| =n}\frac{\log \mu(B_n(x, \om, \vp))^{-1}}{n}, \ h_\mu^l(x, G_1) = \mathop{\lim}\limits_{\vp \to 0}\mathop{\u{\lim}}\limits_{n \to \infty} \mathop{\inf}\limits_{|\om| =n}\frac{\log \mu(B_n(x, \om, \vp))^{-1}}{n}$$
 \end{definition}
 
 In a sense, $h^u_\mu(x, G_1)$ measures the maximal logarithmic rate of growth of the measures of Bowen balls along all trajectories, while $h^l_\mu(x, G_1)$ measures the smallest logarithmic rate of growth of the measures of balls along all  trajectories.
 
 Denote the topological entropy on $Y$ along a trajectory $\om \in \Sigma_m^+$, by $h(Y, G_1, \om)$. We now show that if  local amalgamated entropy can be estimated from below on a set $Y$ of positive measure, then the amalgamated pressure on $Y$ can be estimated.
  
 \begin{theorem}\label{ule}
Let the semigroup $G$ be generated by the finite set $G_1$, and consider $\mu$ a probability measure on $X$. Assume that $Y \subset X$ with $\mu(Y) >0$, and that $$h^l_\mu(x, G_1) \ge \alpha >0, \ \text{for every} \ x \in Y$$ Then, for every trajectory  $\om \in \Sigma_m^+$,  \  $h(Y, G_1, \om) \ge \alpha$. Moreover,  \ $h^A(Y, G_1) \ge \alpha.$
\end{theorem}

\begin{proof}
Let us take an arbitrary number $\beta < \alpha$. We know that $h^l_\mu(x, G_1) \ge \alpha > \beta, x \in Y$. 
Thus for arbitrary $k \ge 1$ let us define the Borel set in $X$, $$Y_k:= \{y \in Y, \mathop{\u{\lim}}\limits_{n \to \infty} \mathop{\inf}\limits_{|\om| =n}\frac{\log \mu(B_n(y, \om, \vp))^{-1}}{n} > \beta, \forall \vp\in(0, 1/k]\}$$
And for every $k, p \ge 1$ introduce also the Borel set in $X$, $$Y_{k, p}:= \{y \in Y_k,  \mathop{\inf}\limits_{|\om| =n}\frac{\log \mu(B_{n}(y, \om, \vp))^{-1}}{n} > \beta, \forall n \ge p,  \vp \in (0, 1/k]\}$$

But then since $\mu(Y) >0$, and $Y = \mathop{\cup}\limits_{k \ge 1} Y_k$ and  $Y_k = \mathop{\cup}\limits_{p \ge 1} Y_{k, p}$, there exist integers $k, p $ arbitrarily large with $\mu(Y_{k, p}) > \frac{\mu(Y)}{2} >0$. 
Hence if $y \in Y_{k, p}$, then for any  $\om \in \Sigma_m^+$ we have for all integers $n \ge p$ and any $\vp \in (0, \frac 1k)$,  that:
\begin{equation}\label{mubn}
\frac{\log \mu(B_{n}(y, \om, \vp))^{-1}}{n} > \beta.
\end{equation}
Let us fix the trajectory $\om \in \Sigma_m^+$, and consider an open cover $\cal F$ of $Y_{k, p}$ with Bowen balls of type $B_n(x, \om, 1/k)$ along $\om$, with all lengths $n$ being larger than $p$. Since $X$ is a compact metric space and thus second-countable, and  $X$ and $Y_{k, p}$ are also  Lindel\"of, it follows that from any open cover we can extract a countable subcover. Thus choose  $\Gamma \subset \cal F$ to be a countable cover of the set $Y_{k, p}$ with Bowen balls $B_{n}(y, \om, 1/k)$, with all the respective lengths $n$  larger than $p$. Then from above we obtain, $$\mathop{\sum}\limits_\Gamma \exp(-\beta n) \ge \mathop{\sum}\limits_\Gamma \mu(B_{n}(y, \om, 1/k)) \ge \mu(Y_{k, p}) >\mu(Y)/2 >0,$$
for some large  $k, p \ge k_0$.
Hence from the definition of trajectory entropy, we obtain  $h(Y_{k, p}, G_1, \om) \ge \beta$. But $Y = \mathop{\cup}\limits_{k, p \ge 1} Y_{k, p}$ and we showed in Section 2 that $h(Y, G_1, \om) = \mathop{\sup}\limits_{k, p} h(Y_{k, p}, G_1, \om)$. Since $\beta$ is arbitrarily smaller than $\alpha$, we obtain $h(Y, G_1, \om) \ge \alpha$.

Finally  using  (\ref{mubn}) for  $\om \in \Sigma_m^+$, and taking an arbitrary cover of $Y$ with Bowen balls along various trajectories, we obtain  as above,  $h^A(Y, G_1) \ge \alpha$.
 \end{proof}
 

 \begin{lemma}\label{zorn}
 Let the compact metric space $X$ and finitely many continuous maps $f_j: X \to X, 1 \le j \le m$. Consider a trajectory $\om \in \Sigma_m^+$, and $\vp >0$, and a family $\cal F$ of Bowen balls $B_n(x, \om, \vp)$ along $\om$, with $n \ge 1$ variable. Then there exists a subfamily $\cal G \subset \cal F$ which contains mutually disjoint balls, such that $$\mathop{\bigcup}\limits_{B \in \cal F} B \subset \mathop{\bigcup}\limits_{B_n(x, \om, \vp) \in \cal G} B_n(x, \om, 3\vp)$$
 \end{lemma}
 
 \begin{proof}
 The proof follows by applying Zorn Lemma to a partially ordered set of families of Bowen balls in which every chain (totally ordered set) has an upper bound, similarly to Lemma 1 of \cite{CRAS}. Here it is important however that all the Bowen balls in $\cal F$ are along the same trajectory $\om$. In this way, if $B_n(x, \om, \vp) \cap B_p(x', \om, \vp) \ne \emptyset$, then if $n \le p$, we conclude from the triangle inequality that $d(f_{\om_j\ldots \om_0}(x), f_{\om_j\ldots \om_0}(x')) \le 3\vp, \ 0 \le j \le n-1$, so $B_p(x', \om, \vp) \subset B_n(x, \om, 3\vp)$. Thus we  extract a disjointed collection $\cal G \subset \cal F$ s.t $3\cal G$ covers the union of sets of $\cal F$.
 \end{proof}  
 
 \begin{remark}\label{nomad}
\ a) If the collection $\cal F$ from Lemma \ref{zorn} consists of Bowen balls along a \textbf{finite set of trajectories} $\cal A \subset \Sigma_m^+$, then the conclusion still holds, with the factor 3 replaced by $M(\cal A)$ depending on $Card(\cal A)$.

\ b) However, Lemma \ref{zorn} \textbf{does not hold} in general if the Bowen balls in $\cal F$ correspond to an \textbf{infinite set of trajectories} $\cal A\subset \Sigma_m^+$. Indeed, in this case we cannot apply the triangle inequality in the proof of Lemma \ref{zorn}, nor can we adjust it with a finite constant $M$. 

For example, assume that $f_j, 1 \le j \le m$ are hyperbolic on  $\Lambda$, have the same unstable manifold $W^u(x)$ at $x$ and are conformal on $W^u(x), x \in \Lambda$. Bowen balls at $x$ along trajectories are tubular neighbourhoods along stable manifolds at $x$, and $\mathop{\inf}\limits_n\frac{|Df_{\om_n\ldots \om_1}|_{E^u(x)}|}{|Df_{\eta_n\ldots \eta_1}|_{E^u(x)}|}$ may be  0, for $\om \ne \eta$.  Then there is no constant $M>0$ s.t  $B_n(x, \om, \vp) \subset B_n(x, \eta, M\vp)$  for all $\eta, \om \in \Sigma_m^+, n \ge 1$. The widths of tubular neighbourhoods $B_n(x, \om, \vp)$ are not comparable for $\om \in \Sigma_m^+$. In this case Lemma \ref{zorn} is not true.
$\hfill\square$   
\end{remark}

 \begin{theorem}
Let $\mu$ be a  probability measure on $X$ and $G$ a semigroup generated by the finite set $G_1$ of maps on $X$. Let $\alpha >0$ and  a subset $Y \subset X$ so that $$h_\mu^u(x, G_1) \le \alpha, \forall x \in Y$$ Then, for any $\om \in \Sigma_m^+$,  \ $h(Y, G_1, \om) \le \alpha, \ \text{and} \ \ h^A(Y, G_1) \le \alpha.$
 \end{theorem}
 
 \begin{proof}
 Let $\beta > \alpha$ arbitrary. We know that for every $x\in Y$, $h_\mu^u(x, G_1) \le \alpha < \beta$. For arbitrary integers $k, p \ge 1$,  define the Borel set $$Y_{k, p}:= \{y \in Y, \exists  \ (n_s)_{s\ge 1},  n_s \ge p,    \mathop{\sup}\limits_{|\om| =n_s}\frac{\log \mu(B_{n_s}(y, \om, \vp))^{-1}}{n_s} \le \beta, s \ge 1,  \vp \in (0, 1/k]\}$$
 Hence if $y \in Y_{k, p}$, then for the associated sequence $\{n_s\}_{s\ge 1}$ with $n_s \ge p, s \ge 1$, and for any trajectory $\om \in \Sigma_m^+$, we have that for every $s \ge 1$ and $\vp \in (0, 1/k]$,
 \begin{equation}\label{mubns}
 \mu(B_{n_s}(y, \om, \vp)) \ge e^{-n_s\beta}
 \end{equation}
Let us  fix  $\om \in \Sigma_m^+$, integers $k, p$ and let  $\vp \in (0, 1/k]$. Let also $N >p$ arbitrary. Consider the cover $\cal F$ of $Y_{k, p}$ with all the Bowen balls $B_{n_s}(x, \om, \vp)$ for $x \in Y_{k, p}$ and $n_s \ge N >p$ for $s \ge 1$. 
Then since all the balls in $\cal F$ correspond to one trajectory $\om$, we  apply Lemma \ref{zorn}  to obtain a disjointed subcollection $\cal G \subset \cal F$, so that the family $3\cal G$ consisting of  $B_{n_s}(x, \om, 3\vp)$ for all  $B_{n_s}(x, \om, \vp) \in \cal G$,  covers the set $Y_{k, p}$.  We can assume $\cal G$ is countable. 
Then, from (\ref{mubns}) and using that $\cal G$ has mutually disjoint balls,
 $$\mathop{\sum}\limits_{\cal G} e^{-n_s \beta} \le \mathop{\sum}\limits_{\cal G} \mu(B_{n_s}(x, \om, \vp)) < 1$$
 Thus since $Y_{k, p} \subset \mathop{\bigcup}\limits_{B_{n_s}(x, \om, \vp) \in \cal G} B_{n_s}(x, \om, 3\vp)$ and as all the integers $n_s$ are larger than  $N$ which is arbitrarily large,  we obtain from the definition of trajectory pressure that, $h(Y_{k, p}, G_1, \om, 3\vp) \le \beta$. Then if $\vp \to 0$, we get $h(Y_{k, p}, G_1, \om) \le \beta$. Using Theorem \ref{amal} and as $Y = \mathop{\cup}\limits_{k, p} Y_{k, p}$,  one obtains $h(Y, G_1, \om) \le \beta$. But $\beta>\alpha$ is arbitrary, hence $h(Y, G_1, \om) \le \alpha$,  $\forall \om \in \Sigma_m^+$. 
So from Theorem \ref{compeat},  $h^A(Y, G_1) \le \alpha.$
 \end{proof}
 
 We now define the local exhaustive entropy for a probability measure $\mu$ on $X$. 
 
 \begin{definition}\label{localexh}
 The \textbf{local exhaustive entropy} of $\mu$ at $x \in X$ is defined as:
 $$h^{+}_\mu(x, G_1) := \mathop{\lim}\limits_{\vp \to 0} \mathop{\u{\lim}}\limits_{n \to \infty} \frac{\log \mu(B_n^+(x, G_1, \vp))^{-1}}{n}$$
 \end{definition}

 We study first the local exhaustive entropy for general probabilities on $X$,
and then at the end of this Section, we study the case of marginal measures $\nu$ 
on X.

 For \textbf{general probability measures} $\mu$ on $X$, we relate their local exhaustive
entropy with the exhaustive topological entropy on sets $Y$:
 
 \begin{theorem}\label{leexh}
 Let a semigroup of continuous maps on the compact metric space $X$, generated by a finite set $G_1$. Consider a probability measure $\mu$ on $X$, and let a Borel subset $Y \subset X$ such that $\mu(Y) >0$. Assume that $$h^{+}_\mu(x, G_1) \ge \alpha, \forall x\in Y.$$ Then, $h^+(Y, G_1) \ge \alpha.$
 \end{theorem}

 \begin{proof}
 Let  an arbitrary $\beta < \alpha$. For every $x \in Y$, $h^{+}_\mu(x, G_1) \ge \alpha >\beta$.  For arbitrary  integers $k , p \ge 1$, define then the Borel set $$Y_{k, p}:= \{y \in Y,  \frac{\log \mu(B_{n}^+(y, G_1, \vp))^{-1}}{n} \ge \beta, \forall n \ge p,  \vp \in (0, 1/k]\}$$
 We have that $Y = \mathop{\bigcup}\limits_{k, p} Y_{k, p}$, and $Y_{k, p} \subset Y_{k, p+1}, Y_{k, p} \subset Y_{k+1, p}$, so there exists some integer $k_0$ such that for $k, p > k_0$, we have $\mu(Y_{k, p}) > \mu(Y)/2 >0$.
 The definition of $Y_{k, p}$ implies that for any $x \in Y_{k, p}$ and every $n \ge p$ and $\vp \in (0, 1/k]$,
 $$\mu(B_n^+(y, G_1, \vp) \le e^{-n\beta}$$
 Then let an arbitrary $\vp \in (0, 1/k]$ and an integer $N>p$ and a countable cover $\cal G$ of $Y_{k, p}$ with exhaustive balls of type $B^+_n(x, G_1, \vp)$ with all the integers $n$ larger than $N$. 
Then using the above measure estimates, we obtain:
 $$\mathop{\sum}\limits_{\cal G}e^{-\beta n} \ge \mathop{\sum}\limits_{\cal G} \mu(B_n^+(x, G_1, \vp)) \ge \mu(Y_{k, p}) > \mu(Y)/2 >0$$
 But since the cover $\cal G$ is arbitrary and $N$ is arbitrarily large, it follows that $h^+(Y_{k, p}, G_1, \vp) \ge \beta$. 
 Thus by taking $\vp \to 0$, $h^+(Y_{k, p}, G_1) \ge \beta$. So from Theorem \ref{proptexh}, $h^+(Y, G_1) = \mathop{\sup}\limits_{k, p} h^+(Y_{k, p}, G_1) \ge \beta$. But $\beta < \alpha$ is arbitrary, thus $h^+(Y, G_1) \ge \alpha.$
 \end{proof}
 
 In the next Theorem, we assume  the maps $f_j, 1 \le j \le m$ are conformal and expanding on a compact set $X \subset \bb R^D$. This means that $f_j$ is $\mathcal C^1$ and  $Df_j(x) = a_j(x) A_j(x), \forall x \in X, 1 \le j \le m$, where $A_j(x)$ is an isometry on $\bb R^D$ and $a_j(x) \in \bb R$, and $|a_j(x)| > \gamma >1, x \in X, 1 \le j \le m$ (see for eg \cite{Pe}). 
  
 \begin{theorem}\label{upesexh}
 Assume the maps $f_j, 1 \le j \le m$ from $G_1$ are conformal and expanding on a compact set $X \subset \bb R^D$, and let a probability $\mu$ on $X$. Let  $Y \subset X$ and $\alpha>0$. If $$h^{+}_\mu(x, G_1) \le \alpha, \forall x \in Y,$$ then $h^{+}(Y, G_1) \le \alpha.$
 \end{theorem}
 
  \begin{proof}
Let an arbitrary $\beta >\alpha$, and for any $k, p\ge 1$, define the Borel set in $X$, $$Y_{k, p}:= \{y \in Y,  \exists \{n_s\}_{s\ge 1}, n_s \ge p,  \frac{\log\mu(B_{n_s}^+(y, G_1, \vp))^{-1}}{n_s} <\beta, \forall s \ge 1, \vp \in (0, \frac 1k]\}$$
Therefore, if $y \in Y_{k, p}$ and $\vp \in (0, \frac 1k]$ and $s \ge 1$, we have 
\begin{equation}\label{nsb}
\mu(B_{n_s}^+(y, G_1, \vp)) \ge e^{-\beta n_s}
\end{equation}
Now, if all functions $f_j\in G_1$ are conformal and expanding on $X$, then also any composition of them is conformal and expanding on $X$.  Thus the Bowen balls $B_n(x, \om, \vp)$ are usual balls centered at $x$ and of radius $|Df_{\om_n\ldots \om_1}(x)|^{-1}$, for any trajectory $\om \in \Sigma_m^+$. Thus for any $x\in X$, the exhaustive ball $B_n^+(x, G_1, \vp)$ is a union of $m^n$ balls centered at $x$. This implies that $B_n^+(x, G_1, \vp)$ is actually a ball centered at $x$ and of radius $\sup\{|Df_{\om_n\ldots\om_1}(x)|^{-1}, \om \in \Sigma_m^+\}$. 

  Now let us take an arbitrary integer $N>p$ and the cover $\cal F$ of $Y_{k, p}$ with  balls $B_{n_s}^+(x, G_1, \vp)$, for all $x \in Y_{k, p}$ and all $s\ge 1$ such that $n_s > N>p$. From above we know that $B_{n_s}^+(x, G_1, \vp)$ is  a regular ball centered at $x$. Now we  apply Besicovitch Covering Theorem (see \cite{He}) to extract a disjointed subcollection $\cal G \subset \cal F$ so that $$Y_{k, p} \subset \mathop{\bigcup}\limits_{B_{n_s}^+(x, G_1, \vp) \in \cal G} B_{n_s}^+(x, G_1, 5\vp)$$
But all integers $n_s$ corresponding to balls in $\cal G$, are larger than $N >p$.
So from (\ref{nsb}),
$$\mathop{\sum}\limits_{\cal G} e^{-\beta n_s} \le \mathop{\sum}\limits_{\cal G} \mu(B_{n_s}^+(x, G_1, \vp) < 1,$$
as  the balls $B_{n_s}^+(x, G_1, \vp) \in \cal G$ are disjointed. The family $\{B_{n_s}^+(x, G_1, 5\vp), B_{N_s}^+(x, G_1, \vp) \in \cal G\}$  covers then $Y_{k, p}$, and $N$ is arbitrary. So $h^+(Y_{k, p}, G_1, \vp) \le \beta$, and for any $k, p\ge 1$, $h^+(Y_{k, p}, G_1) \le \beta$. But $\beta>\alpha$ is arbitrary, so from Theorem \ref{proptexh},  $h^+(Y, G_1) \le \alpha.$

\end{proof}

For \textbf{marginal measures} $\nu$ on $X$ of ergodic $F$-invariant measures $\hat\mu$,  we estimate the local exhaustive entropy $h_\nu^+(x, G_1)$ and lower amalgamated entropy $h_\nu^l(x, G_1)$.

\begin{theorem}\label{expex}
Let a finite set $G_1$ that generates the  semigroup $G$ acting on a compact metric space $X$, and let an $F$-invariant ergodic measure $\hat \mu$ on $\Sigma_m^+\times X$, with $\mu:= \pi_{1*}\hat \mu$ on $\Sigma_m^+$ and $\nu:= \pi_{2*}\hat\mu$ on $X$. Then, there is a Borel set $A\subset X$, $\nu(A) = 1$, such that, $$h^+_\nu(x, G_1) \le h_\nu^l(x, G_1) \le h(\hat\mu) - h(\mu), \ x \in A.$$
\end{theorem}

\begin{proof}
Since $\hat \mu$ is $F$-invariant ergodic on $\Sigma_m^+\times X$ and $F(\om, x) = (\sigma \om, f_{\om_0}(x))$, then also $\mu = \pi_{1*}\hat\mu$ is $\sigma$-invariant ergodic on $\Sigma_m^+$.
By  Brin-Katok Lemma (\cite{BK}) applied to $\hat \mu$ and $F$, and as the Bowen balls for $F$ are of type $[\om_1\ldots \om_n]\times B_n(x, \om, \vp)$, we know that for $\hat\mu$-a.e $(\om, x) \in \Sigma_m^+\times X$, $$\mathop{\lim}\limits_{\vp \to 0}\mathop{\liminf}\limits_{n \to \infty} \frac{-\log \hat\mu([\om_1\ldots \om_n]\times B_n(x, \om, \vp))}{n} = \mathop{\lim}\limits_{\vp \to 0}\mathop{\limsup}\limits_{n \to \infty} \frac{-\log \hat\mu([\om_1\ldots \om_n]\times B_n(x, \om, \vp))}{n} = h(\hat \mu).$$

And from the Brin-Katok Lemma applied to the shift-invariant ergodic measure $\mu= \pi_{1*}\hat\mu$ on $\Sigma_m^+$, we have for $\mu$-a.e. $ \om \in \Sigma_m^+$,
$$\mathop{\lim}\limits_{\vp \to 0}\mathop{\liminf}\limits_{n \to \infty} \frac{-\log \mu([\om_1\ldots \om_n])}{n} = \mathop{\lim}\limits_{\vp \to 0}\mathop{\limsup}\limits_{n \to \infty} \frac{-\log \mu([\om_1\ldots \om_n])}{n} = h(\mu). $$

Let us fix $\vp>0, n\ge 1$, and define the following Borel set in $\Sigma_m^+\times X$, $$\hat X(n, \vp)= \{(\om, x), h(\hat \mu) - \vp \le -\frac{1}{n'}\log\hat\mu([\om_1\ldots \om_{n'}]\times B_{n'}(x, \om, \vp))\le h(\hat \mu) + \vp, n' \ge n\}$$
And for integers $p>n$, define also the Borel set in $\Sigma_m^+$,
$$\Sigma(p, \vp):= \{\om \in \pi_1(\hat X(n, \vp)), h(\mu) - \vp \le \frac{-\log\mu([\om_1\ldots \om_{p'}])}{p'} \le h(\mu) + \vp, p' \ge p\}$$
But $\hat X(n, \vp) \subset \hat X(n+1, \vp), n \ge 1$, and $\hat X(n, \vp') \subset \hat X(n, \vp)$ if $0< \vp'<\vp$. So from above, for any $\vp>0$ small, there is $n = n(\vp)$ and $p(\vp) > n(\vp)$, with $\hat\mu(\hat X(n, \vp)) > 1-\vp$ and $\mu(\Sigma(p, \vp)) > 1-\vp$.
For $\vp>0$ and integers $p > p(\vp)$, define  the Borel set in $X$, $$X(p, \vp) := \pi_2\hat X(p, \vp)$$
 
  Let now $x \in X(p, \vp)$ arbitrary. Then, from the definition of $\hat X(p, \vp)$, for any  $s\ge p$,
  \begin{equation}\label{afara}
  \begin{aligned}
  \nu(B_s^+(x, G_1, \vp)) &\ge  \mathop{\sup}\limits_{|\om| =s} \nu(B_s(x, \om, \vp)) \ge \mathop{\sum}\limits_{\om \in \Sigma(p, \vp), \  [\om_1\ldots \om_s] disjointed} \hat\mu([\om_1\ldots \om_s]\times B_s(x, \om, \vp))\\
  &\ge \mathop{\sum}\limits_{\om \in \Sigma(p, \vp), \ [\om_1\ldots \om_s] disjointed} e^{-s(h(\hat\mu) + \vp)}
  \end{aligned}
  \end{equation}
  Moreover,  notice that \ $\mathop{\sum}\limits_{\om \in \Sigma(p, \vp), \  [\om_1\ldots \om_s] disjointed} \mu([\om_1\ldots \om_s]) \ge \mu(\Sigma(p, \vp)) >1-\vp$. 
  Let us denote by $N'(p, \vp)$ the number of terms in the last sum, i.e the number of disjoint cylinders $[\om_1\ldots \om_s]$ with $\om \in \Sigma(p, \vp)$. 
  But, from the  definition of $\Sigma(p, \vp)$, we have that for each of the terms in the last sum, 
$\mu([\om_1\ldots \om_s]) \le e^{-s(h(\mu)-\vp)}$. 
Hence,
$$N'(p, \vp) \ge (1-\vp) e^{s(h(\mu) -\vp)}$$
Thus, from the above inequality and recalling  (\ref{afara}), we obtain for any integer $s >p$,
$$\nu(B_s^+(x, G_1, \vp)) \ge \mathop{\sup}\limits_{|\om| =s} \nu(B_s(x, \om, \vp))  \ge N'(p, \vp) \cdot e^{-s(h(\hat\mu) + \vp)}\ge (1-\vp) e^{-s(h(\hat\mu)-h(\mu) + 2\vp)}$$
  This implies that, for every point $x \in X(p, \vp)$,
  \begin{equation}\label{muep}
  \mathop{\liminf}\limits_{s\to \infty} \frac{-\log\nu(B_s^+(x, G_1, \vp))}{s} \le \mathop{\liminf}\limits_{s\to \infty} \frac{ \mathop{\inf}\limits_{|\om| =s} \log\nu(B_s(x, \om, \vp))^{-1}}{s}  \le  h(\hat\mu) - h(\mu) +2\vp.
  \end{equation}
  
 But  for any $n>1$, $\hat\mu(\hat X(p(\vp^n), \vp^n)) \ge 1-\vp^n$. Hence if $X(\hat\mu, \vp):= \mathop{\bigcup}\limits_{n >1}X(p(\vp^n), \vp^n)$, then $\nu(X(\hat\mu, \vp)) \ge \hat\mu(\hat X(p(\vp^n), \vp^n)) \ge 1-\vp^n, n >1$. This implies that $\nu(X(\hat\mu, \vp)) = 1$.
Define the Borel set $X(\hat\mu) := \mathop{\bigcap}\limits_{n >1} X(\hat\mu, \frac{1}{2^n})$; then from above, $\nu(X(\hat\mu)) = 1$. 
Thus from  (\ref{muep}), for every point $x \in A:=X(\hat\mu)$ we obtain that, $$h_\nu^+(x, G_1) \le h_\nu^l(x, G_1) \le h(\hat \mu) - h(\mu). $$
\end{proof}

 \section{Amalgamated pressure and dimension estimates.}
 
 In the case of generating functions which have a \textbf{saddle-type} hyperbolic behavior on a $G$-invariant compact set $\Lambda\subset \bb R^D$, it turns out that the amalgamated pressure $P^A$ of the unstable multi-potential $\Phi^u$ is useful to give estimates for the Hausdorff dimension of the slices of $\Lambda$ with submanifolds transversal to the stable directions of $f_j, 1 \le j \le m$.  We do not assume the existence of any common stable spaces, nor any common unstable spaces for the generator maps in $G_1$. From Propositions \ref{exp}-\ref{exp1}, it follows that the estimates using the amalgamated pressure $P^A$ are in general better than the ones obtained with the usual pressure of $f_j, 1 \le j \le m$. 
 
Let a standard splitting of $\mathbb R^D, D \ge 1$ as $\mathbb R^D = \mathbb R^d \times \mathbb R^{D-d}$ for some $d \in [1, D]$. Recall that for $\theta>0$, the \textbf{standard horizontal $(d, \theta)$-cone} $H$ in $\mathbb R^D$ is the set $$H = \{w = (w_1, w_2) \in \mathbb R^D = \mathbb R^d \times \mathbb R^{D-d}, |w_2| \le \theta |w_1| \}$$ And the \textbf{standard vertical $(D-d, \theta)$-cone} $V$ in $\mathbb R^D$ is the set, $$V= \{w = (w_1, w_2) \in \mathbb R^D = \mathbb R^d \times \mathbb R^{D-d}, |w_1| \le \theta |w_2| \}$$  
 In general, a \textbf{cone} $K \subset \mathbb R^D$ is defined as the image $L(H)$ or $L(V)$ of a standard cone through an invertible linear map $L: \mathbb R^D \to \mathbb R^D$ (\cite{KH}). If $K = L(H)$ then the $d$-dimensional linear space $L(\mathbb R^d \times \{0\})$ is called the \textit{core} of $K$, and $d$ is the \textit{dimension} of $K$. Similarly if $K = L(V)$, then $L(\{0\}\times \mathbb R^{D-d})$ is the \textit{core} of $K$ and $D-d$ is the \textit{dimension} of $K$.
   
 Assume now that all the maps $f_j$ from $G_1$ are $\mathcal C^2$ smooth and injective on a neighbourhood of a $G$-invariant compact set $\Lambda \subset \bb R^D$, and there exists an integer $d\in [1, D)$ and  cones $C_1(x), C_2(x)$ depending continuously on $x \in \Lambda$, whose closures intersect only at $0$, such that $C_1(x)$ is a linear image of a horizontal cone of dimension $d$, and $C_2(x)$ is a linear image of a vertical cone of dimension $D-d$, for any $x \in \Lambda$. Moreover assume that for any $x \in \Lambda$ and $ 1\le j \le m$, $$D_xf_j(C_1(x)) \subset C_1(f_j(x)), \ C_2(f_j(x)) \subset D_x{f_j}(C_2(x)),$$ and there exists a number $\lambda>1$ such that  for any $x \in \Lambda$ and $1 \le j \le m$, $$|D_x f_j(w)| \ge \lambda|w|, w \in C_1(x), \ \text{and}  \ |D_xf_j^{-1}(w')| \ge \lambda |w'|, w' \in C_2(f_j(x)).$$ Then $C_1(\cdot)$ is called a \textbf{$G_1$-unstable cone field}, and $C_2(\cdot)$ is called a \textbf{$G_1$-stable cone field} over $\Lambda$. 
In this case, all the maps $f_j, 1 \le j \le m$ are hyperbolic of saddle type on $\Lambda$ (similar to \cite{KH}), and for any $x \in \Lambda$, the cone $C_1(x)$ contains  the unstable tangent space $E^u(x, f_j)$ of $f_j$ at $x$,  and the cone $C_2(x)$ contains the stable tangent space $E^s(x, f_j)$ of $f_j$ at $x$, for $1 \le j \le m$. 

In general, for a linear map $L: \mathbb R^D \to \mathbb R^D$ and a linear subspace $E\subset \mathbb R^D$, denote the \textbf{minimal expansion} of $L$ on $E$ (according to \cite{New}) by, 
\begin{equation}\label{minexp}
m(L |  E) := \mathop{\inf}\limits_{v \in E, v \ne 0} \frac{|L(v)|}{|v|}
\end{equation}
 Then, in the above setting define \textbf{the unstable multi-potential of $G_1$} on $\Lambda$,
 \begin{equation}\label{unspot} 
 \Phi^u(x) = ( - \log m(Df_1 | E^u(x, f_1)), \ldots, -\log m(Df_m | E^u(x, f_m))), \ x \in \Lambda.
 \end{equation}
 From above,  it follows that $\Phi^u \in \cal C(\Lambda, \bb R^m)$. 
 
We say that a submanifold $\Delta \subset \mathbb R^D$ is \textbf{transversal} to a cone $K\subset \bb R^D$ of dimension $\kappa\ge 1$, if $\Delta$ is a $(D-\kappa)$-dimensional submanifold transversal to the core of $K$. 

 The following theorem gives estimates for the Hausdorff dimension of slices in $\Lambda$, by using the amalgamated pressure of the unstable multi-potential $\Phi^u$.
 
 \begin{theorem}\label{dim}
 Assume that the  functions $f_j, 1 \le j \le m$ from $G_1$ are hyperbolic and injective on a $G_1$-invariant set $\Lambda$, and there exist $G_1$-unstable and $G_1$-stable cone fields $C_1(\cdot)$ and $C_2(\cdot)$ respectively over $\Lambda$. 
 Then,
   for every $x \in \Lambda$ and $r>0$ small and every submanifold $\Delta \subset B(x, r)$ transversal to the cone $C_2(x)$,  $$HD(\Delta \cap \Lambda) \le t^{uA}_{G_1},$$
 where  $t^{uA}_{G_1}$ is  the unique zero of  the amalgamated pressure function $t \to P^A(t\Phi^u, G_1)$ of the unstable multi-potential of $G_1$ over $\Lambda$.
 \end{theorem}
 
 \begin{proof}
 Fix a point $x \in \Lambda$ and a small $r >0$, and denote $W:= \Delta \cap \Lambda$. 
 Since the maps $f_j$ have continuous  $G_1$-unstable and $G_1$-stable cone fields $C_1(\cdot), C_2(\cdot)$ respectively over $\Lambda$, it follows that there exist unstable tangent subspaces $E^u(y, f_j) \subset C_1(y), y \in \Lambda$ (see for eg \cite{KH}), such that $D_y f_j|_{E^u(y, f_j)}$ is a dilation of factor larger than $\lambda >1$, for any $1 \le j \le m$ and $y \in \Lambda$.  
 Then from (\ref{unspot}), every component of $\Phi^u$ is smaller than $-\log \lambda$. 
Thus,  from Theorems \ref{amal} and \ref{oscpressure} it follows that the amalgamated pressure function on $\Lambda$, $t \to P^A(t\Phi^u, G_1)$ is strictly decreasing and it converges to $-\infty$ when $t \to \infty$; therefore this function has a unique zero $t^{uA}_{G_1}$.
  
Now we know that $f_j$ dilates with a factor larger than $\lambda>1$ on the intersection of the unstable cone $C_1(x)$ with a ball $B(x, \vp)$ for $\vp \le r$. Also for $y \in \Lambda$ recall that $$B_n(y, \om, \vp) = \{z \in \Lambda, d(z, y) < \vp,  \ldots, d(f_{\om_{n-1}}\ldots f_{\om_1}(z), f_{\om_{n-1}} \ldots f_{\om_1}(y) ) < \vp\}$$ 
From the $G$-invariance of the stable and unstable cones on $\Lambda$ and as $m(Df_j|E^u(x, f_j)) > \lambda >1, x \in \Lambda, 1 \le j \le m$, and since the functions $f_j$ are $\mathcal C^2$ (so their derivatives are Lipshitz continuous), we obtain a Bounded Distortion property on unstable cones. Namely, there exists a constant $C>0$  so that for any $y \in \Lambda$, $\om \in \Sigma_m^+$,  and  any $n \ge 1$, 
$$\begin{aligned}
 C \cdot m(Df_{\om_1}&| E^u(y, f_{\om_1}))^{-1} \cdot \ldots \cdot m(Df_{\om_n}| E^u(f_{\om_{n-1} \ldots \om_1}(y), f_{\om_n}))^{-1} \ge \\ &\ge  |Df_{\om_1}(z)(v_1)|^{-1}\cdot \ldots \cdot |Df_{\om_n}(f_{\om_{n-1} \ldots \om_1}(z))(v_n)|^{-1},
\end{aligned}$$ for any $z \in B_n(y, \om, \vp)$ and any unitary vectors  $v_1 \in C_1(z), \ldots, v_n\in C_1(f_{\om_{n-1} \ldots \om_1}(z))$.  We apply then the Mean Value Theorem succesively, first for $f_{\om_1}$ on $B(y, \vp)$,  then for $f_{\om_2\omega_1}$ on $ B(f_{\om_1}(y), \vp), \ldots,$ and finally  for $f_{\om_{n}}$ on $ B(f_{\om_{n-1}\ldots \om_1}(y), \vp)$. Thus, since the submanifold $\Delta \subset B(x, r)$ is transversal to the $G_1$-stable cone $C_2(x)$, it follows from above that the intersection of $\Delta$ with any  Bowen ball of type $B_n(y, \om, \vp)$ is contained in a ball of diameter:
\begin{equation}\label{kep}
C\cdot m(Df_{\om_1}| E^u(y, f_{\om_1}))^{-1}\cdot \ldots \cdot m(Df_{\om_n}| E^u(f_{\om_{n-1} \ldots \om_1}(y), f_{\om_n}))^{-1}. 
\end{equation}

Now to compute $t^{uA}_{G_1}$, consider a cover $\cal F$ of $\Lambda$ with balls $B_n(y, \om, \vp)$ for various trajectories $\om \in \Sigma_m^+$ and points $y \in \Lambda$. From Theorem \ref{invset}, the amalgamated pressure $P^A$ can also be computed with covers $\cal F$ of $\Lambda$ where all the Bowen balls in $\mathcal F$ correspond to trajectories of the same length $n$.

Let  an arbitrary  number $t > t^{uA}_{G_1}$. This implies that $P^A(t\Phi^u, G_1) < \beta<0$, for some $\beta<0$ which depends on $t$. However in the expression of $P^A(t\Phi^u, G_1)$, the consecutive sum of $t\Phi^u$ on $(y, \om)$ is given by,
\begin{equation}\label{consumt}
S_n(t\Phi^u)(y, \om) = -t \big( \log m(Df_{\om_1}| E^u(y, f_{\om_1})) + \ldots + \log m(Df_{\om_n}| E^u(f_{\om_{n-1} \ldots \om_1}(y), f_{\om_n})) \big).
\end{equation}

Thus, if $P^A(t\Phi^u, G_1) < \beta<0$, then for any $\vp>0$ small and any fixed integer $n>n(\vp)$, there exists a cover $\cal F$ of $\Lambda$ with Bowen balls $B_n(y, \om, \vp)$ (for various $y \in \Lambda, \om \in \Sigma_m^+$) satisfying, $$\mathop{\sum}\limits_{B_n(y, \om, \vp) \in \cal F} e^{S_n(t\Phi^u)(y, \omega)} < e^{n\beta}$$
Recalling that $W = \Delta \cap \Lambda$ and by using (\ref{kep}) and   (\ref{consumt}), we obtain, for any integer $n> n(\vp)$, the following inequality:
\begin{equation}\label{diamW}
\mathop{\sum}\limits_{\cal G} diam\big(W\cap B_n(y, \om, \vp)\big)^t \le C^t \cdot \mathop{\sum}\limits_{\cal F} \exp(S_n(t\Phi^u)(y, \om)) \le C^t e^{n\beta} 
\end{equation}

 However  if we take the collection $\cal G$ of intersections of type $\Delta \cap B_n(y, \om, \vp)$ for $B_n(y, \om, \vp) \in \cal F$, then $\cal G$ covers $W$.
Hence we obtain a cover of $W$ with sets of arbitrarily small diameters with the above property. But for $t$ and $\beta$  as above, there exists $n$ large enough such that $C^t e^{n\beta} < 1$. Thus from (\ref{diamW}), it follows that $HD(W) \le t$. Since $t$ was taken  arbitrarily larger than $t^{uA}_{G_1}$, one obtains $$HD(W) \le t^{uA}_{G_1}.$$

Moreover, for any $1 \le j \le m$ and any real number $t$, notice that 
 $$P^A(t\Phi^u, G_1) \le P(t\Phi^u, G_1, (j, j, j, \ldots)) = P(t\phi_j, f_j)$$ Thus if $t^u_j$ is the unique zero of  the (usual) pressure function   $t \to P(-t\log m(Df_j| E^u(x, f_j)))$ with respect to $f_j$, then we obtain $t^{uA}_{G_1} \le t^u_j$, for any $1 \le j \le m$.
 
 \end{proof}
   
Examples of such semigroups with stable and unstable cone fields can be obtained by taking the generator set $G_1 =\{id_{[0, 1]\times [0, 1]}, f_1, f_2\}$, for skew-products $f_i(x, y) = (g(x), h_i(x, y)), x \in I_1 \cup I_2\subset I, y \in I = [0, 1], I_1\cap I_2 = \emptyset$, $i = 1, 2$,  where $g: I_1 \cup I_2 \to I$ is expanding in $x$ and $h_i(x, \cdot)$ contract uniformly in $y$, similar to \cite{M-MZ}. 

\

 \textbf{Acknowledgements:} This  work was supported in part by grant PN III-P4-ID-PCE-2020-2693 from UEFISCDI Romania.

Eugen Mihailescu, \ Institute of Mathematics of the Romanian Academy,  Calea Grivitei 21, Bucharest, Romania.
 
Email: \ Eugen.Mihailescu@imar.ro \ \ \ \ \ Webpage: \  www.imar.ro/$\sim$mihailes

\end{document}